\documentclass{amsart}
\allowdisplaybreaks
\usepackage{color}

\def\pow{\textsf{pow}}

\def\eps{\varepsilon}
\def\ve{\varepsilon}

\def\C{\mathcal{C}}

\newtheorem{theorem}{Theorem}
\newtheorem{lemma}[subsection]{Lemma}

\theoremstyle{definition}

\newtheorem{proposition}[subsection]{Proposition}
\theoremstyle{remark}

\numberwithin{equation}{section}

\begin{document}

\title[Solutions of Schr\"odinger equation with tetrahedral symmetry]{Solutions of Schr\"odinger equations\\ with symmetry in  orientation preserving tetrahedral group}

\author{Ohsang Kwon}
\address{Department of Mathematics, Chungbuk National University, Cheongju, South Korea}
\email{ohsangkwon@chungbuk.ac.kr}
\thanks{Ohsang Kwon was supported by Basic Science Research Program through the National Research Foundation of Korea(NRF) funded by the Ministry of Education(NRF-2020R1I1A3073436)}

\author{Min-Gi Lee}
\address{Department of Mathematics, Kyungpook National University, Daegu, South Korea}
\email{leem@knu.ac.kr}
\thanks{Min-Gi Lee was supported by the National Research Foundation of Korea(NRF) grant funded by the Korea government(MSIT) (No. 2020R1A4A1018190, 2021R1C1C1011867).
}


\date{\today}



\begin{abstract}
We consider the nonlinear Schr\"odinger equation
\begin{equation*}
\Delta u = \big( 1 +\ve V_1(|y|)\big)u -   |u|^{p-1}u
  \quad \text{in} \quad \mathbb{R}^N, \quad N\ge 3, \quad p \in \left(1, \frac{N+2}{N-2}\right).\end{equation*}
The phenomenon of pattern formation has been a central theme in the study of nonlinear Schr\"odinger equations. However, the following nonexistence of $O(N)$ symmetry breaking solution is well-known: if the potential function is radial and nondecreasing, any positive solution must be radial. Therefore, solutions of interesting patterns can only exist after violating the assumptions.

$O(N)$ symmetry breaking solutions have been presented by Wei and Yan \cite{wei_yan_2014}. Symmetry groups of regular polygons describe their solution patterns. Ever since work of Wei and Yan, there have been substantial generalizations but solutions with higher dimensional symmetry has not been constructed. In this study, the existence of nonradial solutions whose symmetry group is a discrete subgroup of $O(3)$, more precisely, the orientation-preserving regular tetrahedral group is shown.

\end{abstract}

\maketitle



\section{Introduction} \label{sec:intro}

We consider the Schr\"odinger equation
\begin{equation} \label{maineq}
\Delta u = \big( 1 +\ve V_1(|y|)\big)u -   |u|^{p-1}u
  \quad \text{in} \quad \mathbb{R}^N, \quad N\ge 3, \quad p \in \left(1, \frac{N+2}{N-2}\right)\end{equation}
parametrized by a small constant $\ve>0$. $V_1:\mathbb{R}^N \rightarrow \mathbb{R}$ is a contribution to the potential function that is bounded and radially symmetric. We fix a three dimensional subspace $\mathbb{R}^3$ of $\mathbb{R}^N$. Our objective is to construct a solution for \eqref{maineq} where the radial symmetry is broken in such a way the solution peaks at the four vertices of a regular tetrahedron embedded in the subspace. The barycenter of the tetrahedron is at the origin and its diameter is sufficiently large, accordingly as $\ve$ is chosen small.

Our study is about questions arising from exploring the symmetry, or breaking it, of solutions to nonlinear Schr\"odinger equations. By the work of  Gidas et al. \cite{GNN} via the moving plane method,
the following nonexistence of $O(N)$ symmetry breaking solution is well-known: if the potential function is radial and nondecreasing in the radial variable, any positive solution must be radial. Therefore, solutions of interesting patterns, such as those whose symmetry group is a discrete subgroup of $O(N)$, can only exist after violating the assumptions. 
These observations pose questions that under which circumstances what variety of solution patterns could appear. 

Wei and Yan \cite{wei_yan_2014} presented a remarkable result for the potential \eqref{intro-asymp} in the below that there are infinitely many symmetry breaking solutions. Specifically, let $x_i=(z_i,\mathbf{0})$, $i=1,\cdots,k$ for some $k$, where
$$z_i= \left( r \cos\left(\frac{2(i-1)\pi}{k}\right), ~r \sin\left(\frac{2(i-1)\pi}{k}\right)\right)\in \mathbb{R}^2 \quad \text{for $i=1,\cdots,k$}.$$
Then for every $k$ greater than a certain $k_0$ and radius $r$ sufficiently large accordingly, there is a solution $u$ that has segregated $k$ peaks over the specified circle.

In this study, we construct a solution with tetrahedral symmetry, specifically, solution of the form $W^h + \phi$ with $\phi$ of small norm, where $W^h=\displaystyle\sum^4_{i=1}  U_0(y-ht_i)$, $t_i$, $i=1,2,3,4$ are the four vertices of a regular tetrahedron, and $h>0$ denotes the parameter for diameter, and ${U}_0$ denotes the unique  {positive} radially symmetric solution (see \cite{Kwong}) of
\begin{equation*}
  \Delta {U}_0 - {U}_0+   ({U}_0)^p = 0, \quad \text{whose maxima occurs at the origin.}
\end{equation*}

The existence of such a solution reveals two key points on this subject. First, a pattern of tetrahedral symmetry, a discrete subgroup of $O(3)$, is presented. Second, having a radial symmetry breaking solution is not necessarily accompanied by the nondecreasing criterion violation {\it to some substantial extent}. By this, we mean that while at $\ve = 0$ no such solution can exist, and once the nondecreasing criterion is violated by adding a perturbation of size $\ve>0$ (see (A1) and (A2)), such a solution exists no matter how small $\ve$ is. This reveals that the  nonexistence of a radial symmetry breaking solution is a bifurcative phenomenon. 

Precise conditions on the first order contribution $V_1$ of the potential function is as follows.
\begin{itemize}
   \item[(A1)] ${V}_1$ is bounded, smooth, and $V_1(y)=V_1(|y|)$.
   \item[(A2)] ${V}_1(y) = \frac{a}{|y|^m} + O\left(\frac{1}{|y|^{m+\theta}}\right)$ as $|y| \rightarrow\infty$ for some $a>0$, $m>0$, $\theta>0$.
 \end{itemize}
In particular, a problem with the potential function has a connection to the nonlinear eigenvalue problem
\begin{equation} \label{eig}
 \Delta v = \lambda^2 \Big(\hat{V}(|y|)v - |v|^{p-1}v\Big) \quad \text{in} \quad \mathbb{R}^N,
\end{equation}
studied by Ambrosetti and Badiale \cite{AB_1998}, Byeon and Lee \cite{by_2013}, Rabinowitz \cite{paul_rabinowitz_1992}, and Stuart \cite{stuart}. Here,
\begin{equation} \label{intro-asymp}
 \hat{V}(y) = V_0+ \frac{a}{|y|^m} + O\left(\frac{1}{|y|^{m+\theta}}\right) \quad \text{as $|y| \rightarrow \infty$}
\end{equation}
for some positive constants $V_0$, $a$, $m$, and $\theta$. If $V_0=1$, ${V^{\lambda}}(y)=\hat{V}\left(\frac{y}{\lambda}\right)$, and $\lambda^m = \ve$, by the change of variables $v(y)=u(\lambda y)$, the nonlinear eigenvalue problem \eqref{eig} is included in our framework \eqref{maineq}, (A1), and (A2).


The research on this matter has been substantially generalized. For brevity, we summarize the generalizations into three directions, and focus on the last direction. The first is to seek results on the system for multi species problem, and the second is to weaken the assumptions on potentials. For three species system, Peng et al. \cite{peng_wang_wang_2019} constructed nonradial solutions, where the potentials are all positive constants (thus, nondecreasing). Their surprising results revealed that the interactions of species can provide a chance to have a symmetry-breaking solution under the exertion of a  nondecreasing and  radially symmetric potential function. For the second direction, Peng and Wang \cite[Theorem 1.1, 1.2]{peng_wang_2013} considered a problem where only one of two potentials violates the nondecreasing criterion. See also Wang et al. \cite{wang_zhao_2017}, Kwon et al. \cite{kwon-lee-lee_2022}, and Long et al. \cite{long_tang_yang_2020}.

Yet another direction is generalizing the solution pattern. Lin and Peng \cite{lin_peng_2014} considered a pattern for the three species problem such that two species peaks over a circle, whereas one species peaks at the origin. See also Kwon et al. \cite{kwon-lee-lee_2022}. In a multispecies problem, peaks of species may or may not overlap. Peng and Wang \cite{peng_wang_2013} generated both cases: solutions of patterns where two species peak in a synchronized manner at shared sites or in a segregated manner at respective sites over a circle. Zhen \cite{zheng_2017} also studied in this direction.

Any result known to authors other than those with peaks over a circle or at the origin, thus of two dimensional in nature, is only the work of Duan and Musso \cite{dm}. Duan and Musso \cite{dm} considered peak points $\bar{x}_j = (\bar{y}_j,\mathbf{0})$, and $\underline{x}_j = (\underline{y}_j,\mathbf{0})$, where $\bar{y}_j$ and $\underline{y}_j$ are in $\mathbb{R}^3$ for $j=1,\cdots,k$. They are
\begin{equation} \label{pts}
\begin{aligned}
 \bar{y}_j &= r \, \left( \sqrt{1-h^2} \cos\left(\frac{2(j-1)\pi}{k}\right), \sqrt{1-h^2} \sin\left(\frac{2(j-1)\pi}{k}\right),h\right)\\
 \underline{y}_j &= r \, \left( \sqrt{1-h^2} \cos\left(\frac{2(j-1)\pi}{k}\right), \sqrt{1-h^2} \sin\left(\frac{2(j-1)\pi}{k}\right),-h\right)
\end{aligned}
\end{equation}
for a parameter $h>0$. The parameters where solutions are searched were set as $h = O\left(\frac{1}{k}\right)$, and $r = O\left(k\log k\right)$ as $k \rightarrow \infty$ and thus the two circles have to be sufficiently close as $k$ increases. 
The solution structure can be described as follows. As seen in \eqref{pts}, in the $\mathbf{R}^3$ subspace, peak points lies on top and bottom circles of a cylinder, instead of being on a sphere.

Existences established under the assumptions where the potential is not radial symmetric are also notable in the study for nonradial solutions. Ao and Wei \cite{ao_wei_2014} studied \eqref{maineq} for a nonradial potential, where the nonlinearity can be further generalized. See also Cerami et al. \cite{cerami_passaseo_solimini_2013}. Solutions that are nonradial only in the first two coordinates have been constructed without the small parameter $\eps$ in del Pino et al. \cite{DWY}.

In this study, we turn back our attention to a scalar Schr\"odinger equation, and present a solution whose symmetry group is a discrete subgroup of $O(3)$, precisely, the orientation-preserving regular tetrahedral group. Including the work of Duan and Musso \cite{dm}, existing solutions have symmetry in a discrete subgroup of $O(2)$.

Although it is a technical matter, in the below we present differences we encounter in generalizing ideas of Wei and Yan \cite{wei_yan_2014} for higher dimensional configurations. The two key ingredients in Wei and Yan \cite{wei_yan_2014} are the use of symmetry in analyzing the associated linearized operator about the backbone profile, and the space subdivision. The first difference is that the backbone profile is of three dimensional, having the tetrahedral symmetry. The associated linearized operator and its kernel are to be analyzed in a way not similar to those associated with $O(2)$. Thus, Lemma \ref{inu0} of the invertibility is the main finding of this study, and this is the place where the key roles of the orientation-preservint group elements are played.
%
%
%
%
For the second matter of subdivision of space, we observe that one of technical obstacles working higher than two dimensions lies in the fact that while a circle can be subdivided into congruent arcs as many as one wants, $2$-sphere cannot be. Indeed, we see pentagons and hexagons alternatively patched together for $C_{60}$ Buckminsterfullerene (or a soccer ball). Another way to put this is while in two dimensions we have a regular $k$-gon for any large positive integer, we do not have a similar analog in higher dimensions. We give further detail below.

In higher dimensions, it is possible to subdivide the entire space into four congruent closed cones. We simply consider a $3$-simplex, particularly a regular tetrahedron whose vertices are $t_i$, $i=1,2,3,4$ specified in \eqref{ti}. Denoting $y=(y_1,y_2,y_3, y')\in \mathbb{R}^3\times \mathbb{R}^{N-3}$, the four cones are specified by hyperplanes as follows:
\begin{align}\begin{cases}\label{defo}
 \C_1=\left\{ y \in \mathbb{R}^N \ \left| \right.  y_2+y_3\ge 0,\ y_1+y_3\ge0,\ y_1+y_2\ge0  \right\},\\ \\
\C_2=\left\{ y \in \mathbb{R}^N \ \left| \right.  y_2+y_3\le 0,\ y_1-y_2\ge0,\ y_1-y_3\ge0  \right\},\\ \\
\C_3=\left\{ y \in \mathbb{R}^N \ \left| \right.  y_1+y_3\le 0,\ y_1-y_2\le0,\ y_2-y_3\ge0  \right\},\\ \\
\C_4=\left\{ y \in \mathbb{R}^N \ \left| \right.  y_1+y_2\le 0,\ y_1-y_3\le0,\ y_2-y_3\le0  \right\}.
\end{cases}\end{align}
At the center of cone $\C_i$ is the half ray emanated from the origin that passes the vertex $t_i$. It can be checked that every point in $\mathbb{R}^N$ is a nonnegative scalar multiple of the convex combination
$$\lambda_1 t_1 + \lambda_2 t_2 + \lambda_3 t_3 + \lambda_4 t_4$$
and $y\in \C_i$ if and only if $\lambda_i$ is the maximum weight. All points in a half ray emanated from the origin shares the same indices of maximum weight. One sees that the interiors of $\C_i$ are pairwise disjoint, and it is not difficult to see that   $\displaystyle\cup_{i=1}^4 \C_i=\mathbb{R}^N$.

Next, the symmetry group of a regular tetrahedron is represented by orthogonal matrices. The $12$ elements of the subgroup of $SO(N)$ are listed in \eqref{Tdef}-\eqref{Adef}, comprising orientation-preserving members in the tegrahedral group. Its structure can be summarized as follows.
\begin{enumerate}
 \item $\{T_1, T_2, T_3, T_4\}$, and $\{T_1, T_5, T_9\}$ forms two subgroups of $4$ and $3$ elements. ($T_1 = I$)
 \item For $k=1,2,3,4$, $T_k\big|_{C_k}$ is bijective from $\C_k$ to $\C_1$, $(T_k)^{-1} = T_k$.
 \item For $k=1,5,9$, $T_k\big|_{C_1}$ is an automorphism for $\C_1$. $\{T_1, T_5, T_9\}$ is a cyclic group with $(T_5)^2 = T_9$, $(T_5)^3 = T_1$, and
$$\{T_1, T_5, T_9\} = (T_2)^{-1}\{T_2,T_6,T_{10}\} = (T_3)^{-1}\{T_3,T_7,T_{11}\}= (T_4)^{-1}\{T_4,T_8,T_{12}\}.$$
In particular, the upper left block matrix of $T_5$ in \eqref{Tdef} is
$$A_5 = \begin{pmatrix}
         0 & 1 &0 \\ 0 & 0 & 1\\ 1 & 0 & 0
        \end{pmatrix}
      = \begin{pmatrix}
         \tfrac{1}{\sqrt{3}} & -\tfrac{1}{\sqrt{2}} & -\tfrac{1}{\sqrt{6}} \\ \frac{1}{\sqrt{3}} & \tfrac{1}{\sqrt{2}} & -\tfrac{1}{\sqrt{6}}\\ \frac{1}{\sqrt{3}} & 0 & \tfrac{\sqrt{2}}{\sqrt{3}}
        \end{pmatrix}
        \begin{pmatrix}
         1 & 0 &0 \\ 0 & -\frac{1}{2} & \frac{\sqrt{3}}{2}\\ 0 & -\frac{\sqrt{3}}{2} & -\frac{1}{2}
        \end{pmatrix}
        \begin{pmatrix}
         \tfrac{1}{\sqrt{3}} & \tfrac{1}{\sqrt{3}} & \tfrac{1}{\sqrt{3}} \\ -\tfrac{1}{\sqrt{2}} & \tfrac{1}{\sqrt{2}} & 0\\ -\tfrac{1}{\sqrt{6}} & -\tfrac{1}{\sqrt{6}} & \tfrac{\sqrt{2}}{\sqrt{3}}
        \end{pmatrix},$$
    or $A_5$ is a $\frac{2\pi}{3}$ rotation in $\mathbb{R}^3$ about one dimensional subspace passing the vertex $t_1$.
\end{enumerate}
The $12$ symmetries play a significant role in proving the invertibility in Lemma \ref{inu0}. In particular, $\{T_1, T_2, T_3, T_4\}$ lets us symmetrize functions with respect to domains $\C_i$, $i=1,2,3,4$ so that estimates are systematically obtained. Notably, among functions possesing symmetry $\{T_1, T_5, T_9\}$, the kernel of the linearized operator in \eqref{l0} is one dimensional, which can be taken care of.

Once the kernel is shown to be one dimensional, the reduction method is applied. In the work of Wei and Yan \cite{wei_yan_2014}, the number of peaks becomes very large as the peaks are placed far away from the origin. Precisely, the radius $r = O(k\log k)$ as $k \rightarrow \infty$. Heuristically, the small parameter $\ve$ compensates for the shortage of the number of peaks. With the interpretation of the nonlinear eigenvalue problem \eqref{eig}, considering the scaling $v(y)=u(\ve^{\frac{1}{m}} y)$, the peak of $v$ becomes a plateau as $\ve \rightarrow 0$, and the four individual peaks $U_0\big(\ve^{\frac{1}{m}}(y-ht_i)\big)$ have large mass. 

Now, we state our main theorem. The interval $S_\ve$, and constant $\gamma>0$ in the statement of Theorem \ref{mainthm} are described in Sections \ref{notions} and Section \ref{sec:result}.
\begin{theorem} \label{mainthm} Assume the potential function $V(|y|) = 1 + \ve V_1(|y|)$ with $V_1$ satisfying (A1) and (A2). Then $\exists \ve_0$ such that for $\ve\in (0,\ve_0)$, there exists $h\in S_\ve$ and a solution of the form $W^{h} +\phi$ of \eqref{maineq}. We have $\|\phi\|_{H^1}\le C\ve^{\gamma}$ for some $C>0$.
\end{theorem}

The remainder of this article is organized as follows. In Section \ref{notions}, we introduce some notations. We prove the main theorem in Section \ref{sec:result}.

\section{Notations} \label{notions}

First, we let ${U}_0$  be  the unique  {positive} radially symmetric solution (see \cite{Kwong}) of
\begin{equation*}
  \Delta {U}_0 - {U}_0+   ({U}_0)^p = 0,
\end{equation*}
whose maxima occurs at the origin. There exists a constant $\alpha>0$ satisfying (for example, see \cite{AMN})
\begin{equation}\label{exp0}     {\lim_{|y|\to\infty}\left( e^{|y|} |y|^{ \left(\frac{N-1}{2}\right)}{U}_0(|y|)\right)=\alpha, } \end{equation}
   {and thus there is a constant $M>0$ satisfying}
\begin{equation}\label{exp} {U}_0(|y|) \le Me^{- |y|}\min\{|y|^{-\left(\frac{N-1}{2}\right)},1\}\quad   {\textrm{for any}\ y\in\mathbb{R}^N.}\end{equation}

Let $t_i\in \mathbb{R}^3\times \mathbb{R}^{N-3}$ be the vertices of a tetrahedron such that
\begin{equation}\begin{cases}\label{ti}t_1=(~1,~1,~1,~0,\cdots,0), \\  t_2=(1,-1,-1,0,\cdots,0), \\ t_3=(-1,1,-1,0,\cdots,0),\\ t_4=(-1,-1,1,0,\cdots,0),\end{cases}\end{equation}and
 \[U_{h,i}(y):={U}_0(y-{h{t_i}})\quad\mbox{for}\  i=1,\cdots,4,\] where \begin{equation}h\in S_\ve:= \left[ \left(\frac{1}{2\sqrt{2}}-\beta_0\right)\ln\frac{1}{\varepsilon},  \left(\frac{1}{2\sqrt{2}}+\beta_0\right)\ln\frac{1}{\varepsilon}\right].\end{equation} Here,  the constant $\beta_0\in\left(0,\frac{1}{2\sqrt{2}}\right)$ is a small constant.

The subdivision of $\mathbb{R}^N$ into $\C_i$, $i=1,2,3,4$ is as in \eqref{defo}. As seen in \eqref{defo}, six hyperplanes subdivide $\mathbb{R}^N$. For each $\{i,j\}$, $i\ne j$, we denote the plane $P_{ij}(=P_{ji})$ the one contains the midpoint $ \frac{1}{2}(t_i + t_j)$ and the two vertices $\{t_1, t_2, t_3, t_4\} \setminus \{t_i, t_j\}$. The boundary of $\C_i$ consists of three planes $P_{ij}$, $j\in \{1,2,3,4\}\setminus\{i\}$.




    {The linear transformation defined by elements in the regular tetrahedron symmetry group maps the regions ${\C}_i$ to each other.
Let $\mathbf{G}=\{T_{i}\  |\  1\le i \le12\}$} be a subgroup of regular tetrahedron symmetry group (see \cite{XC}), where $T_{i}$ is given by
\begin{align}\label{Tdef}T_{i}=\left(\begin{array}{ll}A_{i} & 0_{3,N-3} \\ 0_{N-3,3} & I_{N-3}\end{array}\right),
 \end{align} where $0_{3,N-3}$ is the $3\times (N-3)$ zero matrix,  $0_{N-3,3}$ is the $(N-3)\times 3$ zero matrix,   $I_{N-3}$ is the  identity matrix of size $(N-3)$, and
 \begin{equation} \label{Adef}
\begin{aligned}
 A_1&=\left(\begin{array}{lll}1&0 &0\\0 &1&0\\ 0&0&1\end{array}\right),
&A_5&=\left(\begin{array}{lll}0&1&0\\0 &0&1\\ 1&0&0\end{array}\right),
&A_9&=\left(\begin{array}{lll}0&0 &1\\1 &0&0\\ 0&1&0\end{array}\right),
\\
A_2&=\left(\begin{array}{lll}1&   0 &~ 0\\0 &-1&~ 0\\ 0&   0&-1\end{array}\right),
&A_6&=\left(\begin{array}{lll}0&1 &~ 0\\0 &0& -1\\ -1&0&~0\end{array}\right),
&A_{10}&=\left(\begin{array}{lll}0&~ 0 &1\\-1 &~0&0\\ 0& -1&0\end{array}\right).
\\
A_3&=\left(\begin{array}{lll}-1&0 & ~0\\0 &1& ~0\\ 0&0&-1\end{array}\right),
&A_7&=\left(\begin{array}{lll}~0&-1 &0\\~0 &~0&1\\ -1& ~0&0\end{array}\right),
&A_{11}&=\left(\begin{array}{lll}0&~0 & -1\\1 &~0& ~0\\ 0&-1&~0\end{array}\right),
\\
A_4&=\left(\begin{array}{lll}-1& 0 &0\\0 &-1&0\\ 0& 0&1\end{array}\right),
&A_8&=\left(\begin{array}{lll}0& -1&~0\\0 &0&-1\\ 1& 0&~0\end{array}\right),
&A_{12}&=\left(\begin{array}{lll}0&   0 & -1\\-1 &0&~ 0\\ 0&   1&~0\end{array}\right).
\end{aligned}
\end{equation}
   {We note that  $\mathbf{G}$ has  orientation-preserving symmetries; thus, $\textrm{det}(T_i)=1$ for $i=1,\cdots,12$. This symmetry property will be employed to extend a function defined on ${\C}_1$ to some well-defined function on other regions ${\C}_i$. Moreover, the   group structure of $\mathbf{G}$ is essential to prove the nondegenercy of a perturbed linearized operator (see  the proof of Lemma \ref{inu0}).
 We refer to \cite{XC} (or Appendix \ref{AA}) for  the multiplication table for the group $\mathbf{G}$. }

 Because $\mathbf{G}$ is a symmetry group of the tetrahedron, we have
 $ T_k^{-1} \,{\{t_1,t_2,t_3,t_4\}} =  \{t_1,t_2,t_3,t_4\}$. We define for each $k=1,\cdots,12$ and $i=1,2,3,4$
 \begin{equation} \label{tki}
  t_{k_i} := T_k^{-1} \,t_i.
 \end{equation}

The norm of $H^1( \mathbb{R}^N)$ is defined as follows:
 $$\|\cdot\|:=\sqrt{\left<\cdot,\cdot\right>},$$ where
$$  \left<u,v\right>:= \int_{ \mathbb{R}^N}\left( \nabla u\cdot\nabla v + uv\right)dy.$$
 We fix a closed subspace $H_s$ of $H^1( \mathbb{R}^N)$ possesing the following symmetry:
\begin{equation*}\begin{aligned}
 H_s:= \Big\{ u \in H^1 (\mathbb{R}^N) ~\Big|~ & u(T_{i}   y) = u(y) \ \mbox{for}\  1\le i\le 12, \\
 &  \mbox{and}\  u(y)=u(y_1,\cdots, y_N) \ \mbox{is even in}\   y_n, 3< n\le N.\Big\}\end{aligned}
\end{equation*}
We  define the function $   {{W}_{h}}$ as follows:
\begin{equation}\label{defw}   {{W}_{h}}(y):={\sum^4_{i=1}} U_{h,i}(y)={\sum^4_{i=1}}  U_0(y-ht_i).\end{equation}
It is easy to see that $   {{W}_{h}}\in H_s$ because the    {linear transformation} definded by $T_i$ is bijective and  $\Pi_{N-3}(t_i)=(0,\cdots,0)\in \mathbb{R}^{N-3}$, $i=1,2,3,4$, where $\Pi_{N-3}:\mathbb{R}^3\times \mathbb{R}^{N-3}\rightarrow \mathbb{R}^{N-3}$ is the projection map defined by $\Pi_{N-3}(y_1, y_2,y_3, y_4,\cdots,y_{N})=(y_4,\cdots,y_{N})$. 
   {Notably,  the function $W_h$ also satisfies}
\begin{equation}\label{eq_wh}   {\Delta W_h - W_h+\sum_{i=1}^4 U_{h,i}^p=0.}\end{equation}

Now we  define a  closed subspace ${E}_h$, where
\begin{equation}
 \begin{aligned}
    {E}_h:= \left\{ \phi \in H_s ~\Big|~  \int_{\mathbb{R}^N}\sum_{i=1}^4  U_{h,i}^{p-1} \frac{\partial U_{h,i}}{\partial h} \,\phi~dy = 0 \right\},
 \end{aligned}
\end{equation}
 and we equip ${E}_h$ with the norm $\|\cdot\|$. We find it convenient to introduce a radial function
 \begin{equation}\label{deff}f(r)=-\frac{U_0(r)^{p-1} U_0'(r)}{r} \quad \quad \quad  \text{for $r>0$}\end{equation}
so that for $i=1,2,3,4$ \begin{equation*} U_{h,i}^{p-1} \frac{\partial U_{h,i}}{\partial h}
= f(|y-ht_i|)(y-ht_i)\cdot t_i.\end{equation*}
We define $\varphi^*_h:=\displaystyle\sum_{i=1}^4 U_{h,i}^{p-1} \frac{\partial U_{h,i}}{\partial h}$. For later purpose, we verify that $\varphi^*_h$ is in $H_s$.
\begin{lemma} $\varphi^*_h \in H_s.$
\end{lemma}
\begin{proof}
$\displaystyle\sum_{i=1}^4  U_{h,i}^{p-1}(y) \frac{\partial U_{h,i}(y)}{\partial h}$ is even in $y_i$, $3<i\le N$, where $y=(y_1,\cdots,y_N)$. For each $k\in [1,12]$ integer, using notations in \eqref{tki} we have
\begin{align*}&\sum_{i=1}^4  U_{h,i}^{p-1} (T_k y)\frac{\partial U_{h,i}(T_k y)}{\partial h}\\ &=\sum_{i=1}^4 f(|T_k y-ht_i|)(T_ky-ht_i)\cdot t_i
\\&=\sum_{i=1}^4 f(|T_k( y-ht_{k_i})|)(T_k(y-ht_{k_i}))\cdot T_k(t_{k_i})
\\&=\sum_{i=1}^4 f(|y-ht_{k_i}|)(y-h_nt_{k_i})\cdot t_{k_i}
=\sum_{i=1}^4  U_{h,i}^{p-1} (y)\frac{\partial U_{h,i}(y)}{\partial h}.\end{align*}
\end{proof}
Notably, $\frac{\partial W_{h}}{\partial h}$ is an element of $H_s$ that is in the kernel of the following linearized problem:
\begin{equation}\label{peq_wh}   {
\Delta  \left(\frac{\partial W_{h}}{\partial h} \right) -\frac{\partial W_{h}}{\partial h}+p\left(\sum_{i=1}^4 (U_{h,i})^{p-1}\frac{\partial U_{h,i}}{\partial h}\right)=0.}\end{equation}
Thus,
$$\left<\frac{\partial W_{h}}{\partial h}, \phi\right>= p\int_{\mathbb{R}^N}\sum_{i=1}^4  U_{h,i}^{p-1} \frac{\partial U_{h,i}}{\partial h} \,\phi~dy$$
and thus $\phi$ being $L^2$-orthogonal to $\varphi_h^*$ is the same as $\phi$ being $H^1$-orthogonal to $\frac{\partial W_{h}}{\partial h}$. We write
$$E_h = \textsf{span} <\frac{\partial W_{h}}{\partial h}>^{\bot} \quad\textrm{in}\ \ H_s.$$

In the next section, we seek a solution  for \eqref{maineq}  with the form $     {{W}_{h}} + \phi$, where $\phi\in {E}_h$ is the perturbation   with small norm.

\section{Results} \label{sec:result}

The scheme to find a solution for \eqref{maineq} is based on the following observations. Suppose $   {{W}_{h}} + \phi $ is a solution of \eqref{maineq}, or $\phi$ formally solves
\begin{equation*}
\Delta \phi -{V_{\varepsilon}}(y) \phi +p \left(   {{W}_{h}}\right)^{p-1} \phi = {g}_{\ve,h}(\phi),
\end{equation*}
where $V_\ve(y) = 1 + \ve V_1(y)$ and
\begin{equation}\begin{aligned}\label{defg}
      {g}_{\ve,h}(\phi):=&  ({V_{\varepsilon}}-1)   {{W}_{h}}-  \left\{ \left|{{W}_{h}}+\phi\right|^{p-1}\left({{W}_{h}}+\phi\right)  -\sum_{i=1}^4 \left(U_{h,i}\right)^p - p \left(   {{W}_{h}}\right)^{p-1}\phi  \right\}.\end{aligned}
\end{equation}
Based on these observations, for a fixed $\phi \in {E}_h$, we consider the following linear functional $\ell_\phi$ on ${E}_h$ that is bounded:
$$\ell_\phi[\psi] := \int_{ \mathbb{R}^N}\left( \nabla \phi\cdot\nabla\psi + V_\ve \phi\psi - p  \left(   {{W}_{h}}\right)^{p-1} \phi\psi\right)dy.$$
This, in turn, via Riesz representation theorem, defines the linear operator $L_{\varepsilon, h}: {E}_h \rightarrow  {E}_h$ by the defining relation
\begin{equation}\label{l0}\left<L_{\varepsilon, h}(\phi),\psi\right>  := \ell_\phi[\psi]= \int_{ \mathbb{R}^N}\left( \nabla \phi\cdot\nabla\psi + V_\ve\phi\psi - p  \left(   {{W}_{h}}\right)^{p-1} \phi\psi\right)dy.\end{equation}
It is not difficult to see the following estimation for $L_{\ve,h}$.
\begin{lemma}   \label{inu00} There is a constant $C >0$, independent of $\ve>0$, such that for any $h\in S_\ve$,
$$     \|L_{\varepsilon, h} \phi\| \le C \|\phi\|    \quad \text{for all $\phi \in E_h$} .$$
\end{lemma}
The next lemma shows that $L_{\ve,h}$ is invertible in $E_h$
\begin{lemma}   \label{inu0} There are  constants $\rho_0 >0$ and $\ve_0>0$ satisfying if $0<\ve<\ve_0$ and $h\in S_\ve$, then
$$    \rho_0 \|\phi\| \le  \|L_{\varepsilon, h} \phi\|    \quad \text{for all $\phi \in E_h$} .$$
\end{lemma}
\begin{proof}
We argue by contradiction. Suppose that there are   $\ve_n \in (0,\ve_0)$, $h_n\in S_{\ve_n}$, and $\phi_n\in E_{h_n}$ with   $\lim_{n\to+\infty}\ve_n=0$ and
\begin{equation*}\label{assume1}\|L_{\ve_n,h_n}(\phi_{n})\|=o_n(1)\|\phi_n\|.\end{equation*} Here, we use the notation $o_n(1)$ to denote  $o_n(1)\to0$ as $n\to+\infty$.
By $(A1)$, $\ve_0$ is chosen so that $\inf_y \{1+\ve V_1(|y|)\}\ge c_0$ for some $c_0>0$ for any $\ve \in(0,\ve_0)$.
We may assume that \begin{equation}\label{assume2}\|\phi_n\|=1.\end{equation}
For simplicity,  we denote $L_{\ve_n,h_n}$, $S_{\ve_n}$,  and $E_{h_n}$ by $L_n$, $S_n$, and $E_n$, respectively. Then, we have
\begin{equation}\label{small_in}<L_{n}(\phi_{n}),\ \psi>=o_n(1)\|\phi_n\|\|\psi\| \ \textrm{for any}\ \psi\in E_n.
\end{equation}
Because $\phi_n(y)=\phi_n(T_i y)$ and $\psi(y)= \psi(T_i y)$ for $1\le i\le 12$, from \eqref{small_in}, we have
\begin{equation}\begin{aligned} \label{by_sym}
&\int_{{\C}_1}\left( \nabla \phi_n\cdot\nabla\psi + V_\ve\phi_n\psi - p  \left(   {{W}_{h}}\right)^{p-1} \phi_n\psi\right)dy
\\&=\frac{1}{4}\int_{ \mathbb{R}^N}\left( \nabla \phi_n\cdot\nabla\psi + V_\ve\phi_n\psi - p  \left(   {{W}_{h}}\right)^{p-1} \phi_n\psi\right)dy\\&=o_n(1)\|\phi_n\|\|\psi\| \ \ \textrm{for any}\ \psi\in E_n.
\end{aligned}\end{equation}
By choosing $\psi=\phi_n$ and using \eqref{assume2},  we also obtain
\begin{equation}\begin{aligned} \label{nonlinear_sym}
&\int_{{\C}_1}\left( |\nabla \phi_n|^2  + V_\ve\phi_n^2 - p  \left(   {{W}_{h}}\right)^{p-1} \phi_n^2\right)dy
=o_n(1),\end{aligned}\end{equation}
and \begin{equation}\begin{aligned} \label{norm_sym}
&\int_{{\C}_1}\left( |\nabla \phi_n|^2  + \phi_n^2\right)dy =\frac{1}{4}.
\end{aligned}\end{equation}
Let \begin{equation}\label{barphi}\bar{\phi}_n(y)=\phi_n(y+h_n t_1).\end{equation}
For any fixed constant $R>0$,  if $n$ is sufficiently large, $B_{R}(h_nt_1)\subset {\C}_1$ because $(x_1,\cdots, x_N)\in B_{R}(h_nt_1)$ implies that  $0<h_n-R\le x_i \le h_n+R$   for $i=1,2,3$ by $h_n\in S_n=   \left[ \left(\frac{1}{2\sqrt{2}}-\beta_0\right)\ln\frac{1}{\varepsilon_n},  \left(\frac{1}{2\sqrt{2}}+\beta_0\right)\ln\frac{1}{\varepsilon_n}\right]$.
Thus, in view of \eqref{norm_sym}, we have
\begin{equation*}\label{bdd1}\int_{B_R(0)}\left( |\nabla \bar{\phi}_n|^2  + \bar{\phi}_n^2\right)dy\le \int_{\{y | y+h_nt_1\in{\C}_1\} } \left( |\nabla \bar{\phi}_n|^2  + \bar{\phi}_n^2\right)dy= \frac{1}{4}.\end{equation*}
Then there exists $\bar{\phi}\in H^1(\mathbb{R}^N)$ such that as $n\to+\infty$,
\begin{equation}\label{con_phi}\bar{\phi}_n\rightharpoonup\bar{\phi}\  \textrm{weakly in}\  H^1_{\textrm{loc}}(\mathbb{R}^N),\  \ \textrm{and}\  \  \bar{\phi}_n\rightarrow\bar{\phi}\  \textrm{strongly in}\  L^2_{\textrm{loc}}(\mathbb{R}^N).\end{equation}

Define
\begin{equation*}\label{def_e} E=\left\{\bar\psi\in H^1(\mathbb{R}^N)\ \Big| \ \int_{\mathbb{R}^N}f(|y|) \left(y \cdot t_{1}\right)  \bar\psi( y)dy=0\right\},
\end{equation*}
where $\displaystyle f(|y|) \left(y \cdot t_{1}\right)=U_0(y)^{p-1} \sum_{i=1}^3\frac{\partial U_0(y)}{\partial y_i}$ since $y\cdot t_1 = \displaystyle\sum_{i=1}^3 y_i$. Notably, $\displaystyle\sum_{i=1}^3\frac{\partial U_0(y)}{\partial y_i}$ is in the kernel of the following linearized problem
\begin{equation}\label{ker}\Delta\left(\sum_{i=1}^3 \frac{\partial U_{0}(y)}{\partial y_i}\right)  -\left(\sum_{i=1}^3 \frac{\partial U_{0}}{\partial y_i}\right)+pU_0(y)^{p-1}\left(\sum_{i=1}^3\frac{\partial U_0(y)}{\partial y_i}\right)=0\end{equation}
and $E=\textsf{span}<\sum_{i=1}^3\frac{\partial U_0(y)}{\partial y_i}>^{\bot}$ in $H^1(\mathbb{R}^N)$. 

First, we claim that $\bar\phi \in E$. In view of  $\phi_n\in E_n$, we have \begin{equation}\begin{aligned}\label{sum_rn}0&=\int_{\mathbb{R}^N}\sum_{i=1}^4  U_{h_n,i}^{p-1} \frac{\partial U_{h_n,i}}{\partial h_n} \,\phi_n~dy
\\&=\sum_{k=1}^4\int_{{\C}_k}\left(\sum_{i=1}^4 U_{h_n,i}^{p-1} \frac{\partial U_{h_n,i}}{\partial h_n} \,\phi_n\right)dy
\\&=\sum_{k=1}^4\int_{{\C}_k}\left(\sum_{i=1}^4 f(|y-h_nt_i|)(y-h_nt_i)\cdot t_i  \phi_n(y)\right)dy,\end{aligned}\end{equation} where $f(\cdot)$ is defined in \eqref{deff}. Using \eqref{tki},
%
\begin{equation}\begin{aligned}\label{omega1_re} &\int_{{\C}_k}\left(\sum_{i=1}^4 f(|y-h_nt_i|)(y-h_nt_i)\cdot t_i  \phi_n(y)\right)dy
\\&=\int_{{\C}_1}\left(\sum_{i=1}^4 f(|T_k(z-h_nt_{k_i})|)T_k(z-h_nt_{k_i})\cdot T_k(t_{k_i})  \phi_n(T_k z)\right)dz
\\&=\int_{{\C}_1}\left(\sum_{i=1}^4 f(|z-h_nt_{k_i}|)(z-h_nt_{k_i})\cdot t_{k_i}  \phi_n( z)\right)dz
\\&=\int_{{\C}_1}\left(\sum_{i=1}^4 f(|z-h_nt_{i}|)(z-h_nt_{i})\cdot t_{i}  \phi_n( z)\right)dz,\end{aligned}\end{equation} here, we used $T_ka\cdot T_kb=a\cdot b$ for any $a,b\in \mathbb{R^N}$ and  $k=1,2,3,4$ in the second equality of \eqref{omega1_re}.

By \eqref{sum_rn} and \eqref{omega1_re}, we have
\begin{align*}0&=\int_{{\C}_1}\left(\sum_{i=1}^4 f(|y-h_nt_{i}|)(y-h_nt_{i})\cdot t_{i}  \phi_n( y)\right)dy
\\&=\int_{\{y| y+h_nt_1\in {\C}_1\}}f(|y|)\left(y \cdot t_{1}\right) \bar{\phi}_n( y)dy
\\&+\int_{{\C}_1}\left(\sum_{i=2}^4 f(|y-h_nt_{i}|)(y-h_nt_{i})\cdot t_{i}  \phi_n( y)\right)dy.\end{align*}
Notably, if $i=2,3,4$, $\lim_{n\to+\infty}\textrm{dist}({\C}_1, h_nt_i)=+\infty$. Using the exponential decay of $U_0$ in \eqref{exp}, $\|\phi_n\|=1$, and the convergence of $\bar{\phi}_n$ in \eqref{con_phi}, we have
\begin{equation}\begin{aligned}\label{omega1_id2}
 0= \int_{\mathbb{R}^N}f(|y|) \left(y \cdot t_{1}\right)  \bar{\phi}( y)dy.\end{aligned}\end{equation}
thus, the claim follows.

Now we claim that $\bar{\phi}$ satisfies
\begin{equation}\label{eq_bar_phi} \Delta \bar{\phi}- \bar{\phi} +p U_0^{p-1}\bar{\phi}=0\quad\textrm{in}\  \ \mathbb{R}^N.\end{equation}
We will prove \eqref{eq_bar_phi} with the following two steps in order to show that
$$\int_{\mathbb{R}^N} \left(\nabla \bar{\phi}\nabla \bar\psi+ \bar{\phi} \bar\psi-p U_0^{p-1}\bar{\phi}\bar\psi \right)dy=0\ \textrm{ for any }\ \bar\psi\in H^1(\mathbb{R}^N)= E \bigoplus E^{\bot}.$$

\textbf{Step 1.} For any fixed constant $R>0$,  let $\bar\psi\in C_0^\infty(B_R(0)) \cap E$ and be even in $y_n$, $3<n\le N$.  We set \begin{equation}\label{defpsi}\psi_n(y)= \bar\psi\left(y-h_nt_1\right)+\bar\psi\left(T_5\left(y-h_nt_1\right)\right)+\bar\psi\left(T_9\left(y-h_nt_1\right)\right).\end{equation} If $n$ is sufficiently large, then $\psi_n\in C_0^\infty\left(B_R\left(h_n t_1\right)\right) \subset C_0^\infty\left({\C}_1\right)$ and \begin{equation*}\label{on_omega1}\psi_n(y)=\psi_n(T_5 y)=\psi_n(T_9 y)\quad\textrm{for}\ \ y\in {\C}_1.\end{equation*} We extend $\psi_n$ outside ${\C}_1$ to define an element in $E_n$ as follows:
\begin{equation}\label{extend}\psi_n(y):= \psi_n(T_i y)\quad\textrm{on}\ \ {\C}_i, \ i=2,3,4.\end{equation}
We check that $\psi_n \in H_s$ in Appendix \ref{appendix_claim}. We consider
$$c_n := \left<\psi_n, \frac{{\varphi}^*_{h_n}}{\|{\varphi}^*_{h_n}\|_{L^2(\mathbb{R}^n)}} \right> \frac{1}{\|{\varphi}^*_{h_n}\|_{L^2(\mathbb{R}^n)}}, \quad \hat\psi_n = \psi_n - c_n{\varphi}^*_{h_n}.$$
We claim that  $\hat{\psi}_n \in E_n$ for each $n$ and
\begin{equation} \label{appendix1}
\displaystyle\lim_{n\to+\infty}c_n=0.
\end{equation}
Because  $\bar\psi\in C_0^\infty(B_R(0)) \cap E$
$$0=\int_{\mathbb{R}^N}f(|y|) \left(y \cdot t_{1}\right)  \bar\psi( y)dy=\int_{B_R(0)}f(|y|) \left(y \cdot t_{1}\right)  \bar\psi( y)dy,$$ and thus
\begin{equation*}\begin{aligned}\label{psi_sum}
0&= \int_{B_R(h_n t_1)}f(|y-h_nt_1|) \left((y-h_nt_1) \cdot t_{1}\right)  \bar\psi( y-h_nt_1)dy
\\&+\int_{B_R(h_nt_1)}f(|T_5\left(y-h_nt_1\right)|) \left(T_5\left(y-h_nt_1\right) \cdot T_5 (t_{1})\right)  \bar\psi( T_5\left(y-h_nt_1\right))dy
\\&+\int_{B_R(h_nt_1)}f(|T_9\left(y-h_nt_1\right)|) \left(T_9\left(y-h_nt_1\right) \cdot T_9 (t_{1})\right)  \bar\psi( T_9\left(y-h_nt_1\right))dy
\\&= \int_{B_R(h_nt_1)}f(|y-h_nt_1|) \left((y-h_nt_1) \cdot t_{1}\right)
\\&\times\left\{\bar\psi( y-h_nt_1)+  \bar\psi\left( T_5\left(y-h_nt_1\right)\right)+  \bar\psi\left( T_9\left(y-h_nt_1\right)\right)\right\}dy
\\&= \int_{B_R(h_nt_1)}f(|y-h_nt_1|) \left((y-h_nt_1) \cdot t_{1}\right)
\psi_n(y)dy,
 \end{aligned}\end{equation*}here, we used $t_1=T_5(t_1)=T_9(t_1)$ in the first equality,  $T_ka\cdot T_kb=a\cdot b$ for any $a,b\in \mathbb{R^N}$ and  $k=5,9$ in the second equality, and  $\psi_n(y)= \psi\left(y-h_nt_1\right)+\psi\left(T_5\left(y-h_nt_1\right)\right)+\psi\left(T_9\left(y-h_nt_1\right)\right)$ on $B_R(h_n t_1)\subset{\C}_1$ by  \eqref{defpsi} in the third equality.  Because $\psi_n\equiv0$ on ${\C}_1\setminus B_R(h_nt_1)$,  we have \begin{equation*}\label{psi_sum2}
0=  \int_{{\C}_1}f(|y-h_nt_1|) \left((y-h_nt_1) \cdot t_{1}\right)
\psi_n(y)dy \quad \text{and}
 \end{equation*}
\begin{equation*}\label{psi_sum3_a}c_n=\frac{\int_{{\C}_1}\left(\sum_{i=2}^4 f(|y-h_nt_{i}|)(y-h_nt_{i})\cdot t_{i}  \psi_n( y)\right)dy}{\int_{{\C}_1}\left(\sum_{i=1}^4 f(|y-h_nt_{i}|)(y-h_nt_{i})\cdot t_{i}  \right)^2dy}.\end{equation*}Moreover, $\lim_{n\to+\infty}\textrm{dist}({\C}_1, h_nt_i)=+\infty$ for $i=2,3,4$, and the exponential decay of $U_0$ in \eqref{exp} imply  that the denominator
  \begin{equation*}\begin{aligned}&\int_{{\C}_1}\left(\sum_{i=1}^4 f(|y-h_nt_{i}|)(y-h_nt_{i})\cdot t_{i}  \right)^2dy\\&\ge \int_{B_1(0)}\left(f(|z|)(z\cdot t_{1} ) \right)^2dz+o_n(1)\ge c_0>0\ \textrm{for some constant}\ c_0>0,\end{aligned}\end{equation*} and $\lim_{n\to+\infty}c_n=0$.
Using similar arguments in \eqref{sum_rn} and \eqref{omega1_re}, we obtain
 \begin{equation*}\begin{aligned}&\int_{\mathbb{R}^N}\sum_{i=1}^4  U_{h,i}^{p-1} \frac{\partial U_{h,i}}{\partial h} \,\left(\psi_n-c_n\sum_{i=1}^4  U_{h_n,i}^{p-1} \frac{\partial U_{h_n,i}}{\partial h_n}\right)~dy
\\&=4\int_{{\C}_1}\sum_{i=1}^4  U_{h,i}^{p-1} \frac{\partial U_{h,i}}{\partial h} \, \left( \psi_n( y)-c_n\sum_{i=1}^4 f(|y-ht_{i}|)(y-ht_{i})\cdot t_{i} \right)dy=0\end{aligned}\end{equation*}
and the claim follows.

Now, \eqref{by_sym} and \eqref{appendix1} imply that
\begin{align*}
o_n(1)&=\int_{{\C}_1}\left( \nabla \phi_n\cdot\nabla\hat{\psi}_n + V_\ve\phi_n\hat{\psi}_n - p  \left(   {{W}_{h}}\right)^{p-1} \phi_n\hat{\psi}_n\right)dy
\\&=\int_{{\C}_1}\left( \nabla \phi_n\cdot\nabla{\psi}_n + V_\ve\phi_n{\psi}_n - p  \left(   {{W}_{h}}\right)^{p-1} \phi_n{\psi}_n\right)dy
+o_n(1)\\&=\int_{B_R(h_nt_1)}\left( \nabla \phi_n\cdot\nabla{\psi}_n + V_\ve\phi_n{\psi}_n - p  \left(   {{W}_{h}}\right)^{p-1} \phi_n{\psi}_n\right)dy
+o_n(1),\end{align*} here, we used $\psi_n\in C_0^\infty\left(B_R\left(h_n t_1\right)\right)$.
Moreover, the exponential decay of $U_0$ in \eqref{exp}, definiton of $\bar{\phi}_n$ in \eqref{barphi}, and property of $V_\ve$ in (A1) and (A2), we obtain
\begin{align*}
o_n(1)& =\int_{B_R(h_nt_1)}\left( \nabla \phi_n\cdot\nabla{\psi}_n + V_\ve\phi_n{\psi}_n - p  \left(   {{W}_{h}}\right)^{p-1} \phi_n{\psi}_n\right)dy
\\&=\int_{B_R(h_nt_1)}\Big( \nabla\bar{\phi}_n(y-h_nt_1)\cdot\nabla{\psi}_n (y)+ V_\ve\bar{\phi}_n(y-h_nt_1){\psi}_n
\\& - p  \left(U_0(y-h_nt_1)\right)^{p-1} \bar{\phi}_n(y-h_nt_1){\psi}_n\Big)dy+o_n(1)
\\&=\sum_{i=1,5,9}\int_{B_R(0)}\Big( \nabla\bar{\phi}_n(y)\cdot \nabla {\bar\psi}(T_i y)+ \bar{\phi}_n(y){\bar\psi}(T_i y)
\\& - p  \left(U_0(y)\right)^{p-1} \bar{\phi}_n(y){\bar\psi}(T_i y)\Big)dy+o_n(1).\end{align*} Since $\phi_n\in E_n$ and $T_i(h_nt_1)=h_nt_1$ for $i=1,5,9$, we have  $\bar{\phi}_n(y)=\bar{\phi}_n(T_i y)$ for $i=1,5,9$.  Using  \eqref{con_phi} and  $\psi\in C_0^\infty\left(B_R\left(0\right)\right)$, we obtain
\begin{equation}\begin{aligned} \label{by_sym4}
 0&=\int_{\mathbb{R}^N}\left( \nabla\bar{\phi}(y)\cdot \nabla {\bar\psi}( y)+ \bar{\phi}(y){\bar\psi}( y)
 - p  \left(U_0(y)\right)^{p-1} \bar{\phi}(y){\bar\psi}(y)\right)dy.\end{aligned}\end{equation}
However, because $\bar{\phi}(y_1,\cdots,y_N)$ is even in $y_n$, $3<n\le N$, \eqref{by_sym4}  holds for any function $\psi\in C_0^\infty(B_R(0))$, which is odd in $y_n$, $3<n\le N$. Therefore,  \eqref{by_sym4} holds for any $\bar\psi\in E$ because of  the density of $C_0^\infty(B_R(0))$ in $H^1(\mathbb{R}^N)$.

\textbf{Step 2.} Let $\bar\psi=\sum_{i=1}^3 \frac{\partial U_0(y)}{\partial y_i}$. In view of  \eqref{ker}, we have
\begin{equation*}\Delta\bar\psi  -\bar\psi +p(U_0(y))^{p-1}\bar\psi=0.\end{equation*}
Thus, \eqref{by_sym4} also holds for $\bar\psi\in E^{\bot}$.

\bigskip

In view of Steps 1 and 2, we prove the claim \eqref{eq_bar_phi}, that is,
\begin{equation*} \Delta \bar{\phi}- \bar{\phi} +p U_0^{p-1}\bar{\phi}=0\quad\textrm{in}\  \ \mathbb{R}^N.\end{equation*}
Because $U_0$ is nondegenerate (see \cite{Kwong}) and $\bar{\phi}(y_1,\cdots, y_N)$ is even
in $y_n$, $3<n\le N$, there are constants $c_i$, $i=1,2,3$, such that \begin{equation*}\bar{\phi}(y)=\sum_{i=1}^3 c_i\frac{\partial U_0(y)}{\partial y_i}=\frac{U_0'(|y|)}{|y|}\sum_{i=1}^3 c_i  y_i.\end{equation*}
We claim that \begin{equation}\label{ci0}\bar{\phi}\equiv0.\end{equation}  To prove \eqref{ci0}, we first recall \eqref{omega1_id2}, i.e., $
 0= \int_{\mathbb{R}^N}f(|y|) \left(\sum_{i=1}^3y_i\right)  \bar{\phi}( y)dy$. Then we have  \begin{equation}\label{sumc}\sum_{i=1}^3 c_i=0.\end{equation}
Because $\phi_n\in H_s$,  $T_k(h_n t_1)=h_nt_1$ for $k=1,5,9$, and  $\bar{\phi}_n(y)=\phi_n(y+h_n t_1)\rightharpoonup\bar{\phi}(y)$  {weakly in} $H^1_{\textrm{loc}}(\mathbb{R}^N)$ as $n\to+\infty$, we have $\bar{\phi}(y)=\bar{\phi}(T_k y)$ for $k=1,5,9$, implying $c_1=c_2=c_3$.  Together with \eqref{sumc}, we conclude that $c_1=c_2=c_3=0$; thus,  the claim \eqref{ci0} holds.

Consequently,  for any fixed constant $R>0$,  we obtain
\begin{equation}\label{smallp}
\int_{B_R(h_nt_1)}\phi_n^2 dy=o_n(1).\end{equation}

By \eqref{norm_sym}, we have
\begin{align*}
\int_{{\C}_1}\left( |\nabla \phi_n|^2  + V_\ve\phi_n^2  \right)dy
&\ge \min\{1, c_0\}\int_{{\C}_1}\left( |\nabla \phi_n|^2  + \phi_n^2  \right)dy\\&= \frac{\min\{1,c_0\}}{4}>0.
\end{align*}
In view of \eqref{nonlinear_sym}, the exponential decay of $U_0$ in \eqref{exp}, $\lim_{n\to+\infty}\textrm{dist}({\C}_1, h_nt_i)=+\infty$ for $i=2,3,4$,  and \eqref{smallp}, we  obtain
\begin{align*}
o_n(1)&=\int_{{\C}_1}\left( |\nabla \phi_n|^2  + V_\ve\phi_n^2 - p  \left(   {{W}_{h}}\right)^{p-1} \phi_n^2\right)dy
\\&\ge  \frac{\min\{1,c_0\}}{4}-\int_{{\C}_1}p  \left(U_0(y-h_nt_1)\right)^{p-1} \phi_n^2dy+o_n(1)
\\&= \frac{\min\{1,c_0\}}{4}+o_n(1),\end{align*}
which is a contradition.
\end{proof}

We recall from \eqref{defg} that
\begin{equation*}\begin{aligned}
      {g}_{\ve,h}(\phi)=&  ({V_{\varepsilon}}-1)   {{W}_{h}}-  \left\{ \left|   {{W}_{h}}+\phi\right|^{p-1}(W_h+\phi)  -\sum_{i=1}^4 \left(U_{h,i}\right)^p - p \left(   {{W}_{h}}\right)^{p-1}\phi  \right\}.\end{aligned}
\end{equation*}
In view of  Sobolev embedding,  we have $ {g}_{\ve,h}(\phi)$ {which also defines the bounded linear functional ${G}_{\ve,h}(\phi)$ on ${E}_h$ such that}
\begin{align*}
  {G}_{\ve,h}(\phi)[\psi]:=\int_{ \mathbb{R}^N}  {g}_{\ve,h}(\phi)\psi \quad \text{for $\psi \in {E}_h$}.
\end{align*}
Applying the Riesz representation theorem, there exists  $ \Gamma_{\varepsilon,h}(\phi) \in {E}_h$ with
$$ \left< \Gamma_{\varepsilon,h}(\phi),\psi\right>  =  {G}_{\ve,h}(\phi)[\psi] \quad \text{for $\psi \in {E}_h$}.$$
Combined with the  inverse   $L_{\varepsilon, h}^{-1}$ on  ${E}_h$, we can define the operator $F_{\ve,h}:  E_h \rightarrow E_h$ such that  \begin{equation*}\label{defF}F_{\ve,h}(\phi):=L_{\varepsilon, h}^{-1}\Big( \Gamma_{\varepsilon,h}(\phi)\Big)\quad\textrm{for}\ \ \phi\in E_h.\end{equation*} If ${\phi}$ is any fixed point of $F_{\ve,h}$, then the following holds
\begin{align*}
 &{L_{\varepsilon, h}(\phi)} =  \Gamma_{\varepsilon,h}(\phi) \end{align*}
 if and only if
 \begin{align*}
 \int_{\mathbb{R}^N} \nabla \phi \cdot \nabla \psi + {V_{\varepsilon}}(y) \phi\psi - p \left(   {{W}_{h}}\right)^{p-1}\phi\psi = \int_{\mathbb{R}^N} - {g}_{\ve,h}(\phi)\psi \quad \mbox{for all}\ \psi \in {E}_h.
  \end{align*}

To obtain a fixed point for $F_{\ve,h}$ and estimate the energy for \eqref{maineq}, we  establish estimations for the sum $\displaystyle \sum_{i=1}^4  U_{h,i}$  on ${\C}_1$. The proof is motivated by Lemma A.1 in \cite{wei_yan_2014}.
\begin{lemma}\label{ksum}  \cite[Lemma A.1]{wei_yan_2014} Fix any $\eta \in (0,2]$.  Then for any $h\in S_\ve$ and $y \in{\C}_1$, we have
\begin{equation*}\begin{cases}&\displaystyle \sum_{i=2}^4 U_{h,i}(y) \le    {3M e^{-\sqrt{2}\eta h}   e^{-(1-\eta)|y-{ht_1}|}\min\{|y-ht_1|^{-\left(\frac{N-1}{2}\right)},1\},} \\&  \displaystyle    {{W}_{h}}(y)=\sum_{i=1}^4 U_{h,i}(y) \le 4Me^{-|y-{ht_1}|}\min\{|y-ht_1|^{-\left(\frac{N-1}{2}\right)},1\}.\end{cases}\end{equation*}
 \end{lemma}
\begin{proof}
From the exponential decay of $U_0$ in \eqref{exp}, for each $i$, we have
\begin{align*}
& U_{h,i}(y) =U_0(|y-ht_i|)\le M    {\min\{|y-ht_i|^{-\left(\frac{N-1}{2}\right)},1\}}e^{-|y- ht_i|}
\\&= M   {\min\{|y-ht_i|^{-\left(\frac{N-1}{2}\right)},1\}}e^{ -\frac{\eta}{2}| ht_i- ht_1| }e^{ \frac{\eta}{2}| ht_i- ht_1| }e^{ -|y- ht_i| }\\&
 \le M   {\min\{|y-ht_1|^{-\left(\frac{N-1}{2}\right)},1\}}e^{ -\frac{\eta}{2}| ht_i- ht_1| }e^{ \frac{\eta}{2}|y- ht_1| }e^{ ( \frac{\eta}{2}-1)|y- ht_i| }
\\&\le M   {\min\{|y-ht_1|^{-\left(\frac{N-1}{2}\right)},1\}}e^{ -\frac{\eta}{2}| ht_i- ht_1| }e^{ ( -1 + \eta)|y- ht_1| },
\end{align*}
where we employed that $ \frac{\eta}{2} -1 \le 0$ and that $|y- ht_i|\ge |y- ht_1|$ for $y\in {\C}_1$. Therefore, we obtain
\begin{align*} \sum_{i= 2}^4 U_{h,i}(y) &\le M   {\min\{|y-ht_1|^{-\left(\frac{N-1}{2}\right)},1\}}e^{(-1+\eta)|y- ht_1|}\sum_{i=2}^4 e^{-\frac{\eta}{2}| ht_i- ht_1|}\\&= 3 M   {\min\{|y-ht_1|^{-\left(\frac{N-1}{2}\right)},1\}}e^{(-1+\eta)|y- ht_1|}  e^{-\sqrt{2}\eta h }.\end{align*}
We also see that the exponential decay of $U_0$ in \eqref{exp} and $|y- ht_i|\ge |y- ht_1|$ for $y\in {\C}_1$ imply
\begin{align*}\sum_{i= 1}^4 U_{h,i}(y) &\le M\sum_{i= 1}^4e^{-|y- ht_i|} \min\{ |y-ht_i|^{-\left(\frac{N-1}{2}\right)},1\} \\&\le  4Me^{-|y- ht_1|} \min\{ |y-ht_1|^{-\left(\frac{N-1}{2}\right)},1\}.\end{align*}Now we complete the proof of Lemma \ref{ksum}.
 \end{proof}
Let $\gamma:=\frac{1}{2}-\sqrt{2}\beta_0>0$,
and   \begin{equation*}\label{Bgamma}B_{\ve,h}:=\{ \phi \in  E_h \ | \   \|{\phi}\| \le \ve^{\gamma}\}.\end{equation*}
Now we prove the existence of a fixed point of   $F_{\ve,h}$ on $B_{\ve,h}$.
\begin{proposition} \label{Fixed} There exists $\ve_0$ such that for $\ve\in(0, \ve_0)$ and $h \in S_\ve$, the map $F_{\ve,h}$ has a fixed point $\phi_h \in B_{\ve,h}$.
 \end{proposition}
\begin{proof}
\textit{Step 1.} In this step, we show that if  $\ve>0$ is sufficiently small, then    \begin{equation}\label{claimp}F_{\ve,h}(B_{\ve,h}) \subseteq B_{\ve,h}\quad\textrm{ for    any}\  \  {h \in S_{\ve}.}\end{equation}
By Lemma \ref{inu0} and \eqref{defg}, we have
 \begin{equation}\begin{aligned}\label{es}
  \|F_{\ve,h}(\phi)\| &= \left\|L_{\varepsilon, h}^{-1}\Big( \Gamma_{\varepsilon,h}{\phi}\Big)\right\| \le \frac{1}{\rho_0}\left\| \Gamma_{\varepsilon,h}{\phi}\right\| = \frac{1}{\rho_0} \sup_{\|\psi\|=1} \left| \int_{ \mathbb{R}^N}  {g}_{\ve,h}(\phi) \psi \right|
  \\&\le  {\frac{1}{\rho_0}} \sup_{\|\psi\|=1} \Bigg\{   {\int_{ \mathbb{R}^N}  }|\psi| \left|  ({V_{\varepsilon}}-1)   {{W}_{h}}\right| dy\\
  & +  {\int_{ \mathbb{R}^N} } |\psi| \left|  |W_h+\phi|^{p-1}\left(    {{W}_{h}}+\phi\right)  - \left(   {{W}_{h}}\right)^p - p \left(   {{W}_{h}}\right)^{p-1}\phi \right|dy
\\&+ {\int_{ \mathbb{R}^N} } |\psi| \left|  \left(   {{W}_{h}}\right)^p  - \sum_{i=1}^4 \left(U_{h,i}\right)^p \right|dy\Bigg\}.
 \end{aligned}\end{equation}

First, we recall  that $V_\ve(y)=1 + \ve V_1(y)$, and observe that for ${\sigma}\in\left(\sqrt{\frac{2}{3}},1\right)$,
\begin{equation}\begin{aligned} \label{1e}
 & {\int_{ \mathbb{R}^N}  }|\psi| \left|  ({V_{\varepsilon}}-1)   {{W}_{h}}\right| dy \le \sum_{i=1}^4\|\psi\|_{L^2(\mathbb{R}^N)} \left\| ({V_{\varepsilon}} -1)U_{h,i}\right\|_{L^2(\mathbb{R}^N)}
\\
 &=4 \|\psi\|_{L^2(\mathbb{R}^N)}\big\| ({V_{\varepsilon}} -1)U_{h,1}\big\|_{L^2(\mathbb{R}^N)}
 \\&=4\ve \|\psi\|_{L^2(\mathbb{R}^N)}\Big\| V_1(y+ht_1)U_{0}(y)\Big\|_{L^2(\mathbb{R}^N)}
\\
  &\le 4\ve \|\psi\|  \Big\| V_1(y+ht_1)U_{0}(y)\Big\|_{L^2(B_{{\sigma} h|t_1|}(0))}\\&\ +4\ve\|\psi\| \sup
_{y\in \mathbb{R}^N}|V_1(y)| \| U_{0}\|_{L^2(\mathbb{R}^N\setminus B_{{\sigma} h|t_1|}(0))} \\
 &\le C_1 \ve\|\psi\| \Big\{\Big\| V_1(y+ht_1)U_{0}(y)\Big\|_{L^2(B_{{\sigma} h|t_1|}(0))}+e^{-{\sigma} h |t_1|}\Big\}, \quad |t_1| = \sqrt{3},
\end{aligned}\end{equation}where $C_1>0$ is a constant, independent of $\ve$, and $h>0$.
Because ${|y+ht_1|}\ge {(1-{\sigma})h|t_1|}$ on $B_{{\sigma} h|t_1|}(0)$, the assumption $(A2)$ implies that
\begin{equation}\begin{aligned} \label{2e}
&\left\| V_1(y+ht_1)U_{0}(y)\right\|_{L^2(B_{{\sigma} h|t_1|}(0))}
\\&\le C_1 \left\| \left( \frac{1}{|y+ht_1|^{m }}\right)U_{0}(y)\right\|_{L^2(B_{{\sigma} h|t_1|}(0))}
\le \frac{C_2}{(1-{\sigma})^m}\frac{1}{   {{h}^{m }}}
 \end{aligned}\end{equation}where $C_1, C_2>0$ are constants, independent of $\ve, h>0$, and ${\sigma}\in\left(\sqrt{\frac{2}{3}},1\right)$.

Second, we estimate   $ {\int_{ \mathbb{R}^N} } |\psi| \left|  \left|   {{W}_{h}}+\phi\right|^{p-1}\left(   {{W}_{h}}+\phi\right)  -  \left(   {{W}_{h}}\right)^p - p \left(   {{W}_{h}}\right)^{p-1}\phi \right|dy$. For brevity, we introduce a function for $q\ge0$, $\textsf{pow}_q(x)= \begin{cases}
                                                          x^q, & x\ge0,\\
                                                          -(-x)^q, & x<0.
                                                         \end{cases}
$
We consider the following two cases.

\textbf{Case 1.} $2\le p$ : In this case,
\begin{equation} \label{p222}
\begin{aligned}
 &\left|\pow_p(W_h + \phi)  -  \left(   {{W}_{h}}\right)^p - p \left(   {{W}_{h}}\right)^{p-1}\phi\right| \\
 &= \left|p(p-1)\phi^2 \int_0^1 \int_0^s \pow_{p-2}(W_h + \lambda \phi) \: d\lambda ds\right|
 \end{aligned}
\end{equation}
and thus  there are constants $c_1, c_2>0$, independent of $\ve, h>0$,  satisfying
 \begin{equation}\label{p22}
 \begin{aligned}&\int_{\mathbb{R}^N}|\psi|\left| \pow_p({{W}_{h}} +\phi)  -  \left(   {{W}_{h}} \right)^p - p \left(   {{W}_{h}}\right)^{p-1}\phi  \right|
\\&\le c_1 \int_{\mathbb{R}^N}|\psi| \left(|\phi|^2+|\phi|^p\right)dy
\\&\le c_1 \left(\|\psi\|_{L^{p+1}(\mathbb{R}^N)}\|\phi\|_{L^{\frac{2(p+1)}{p}}(\mathbb{R}^N)}^2+ \|\psi\|_{L^{p+1}(\mathbb{R}^N)}\|\phi\|_{L^{p+1}(\mathbb{R}^N)}^p\right)
\\&\le  c_2\|\psi\|\|\phi\|^2,\end{aligned}\end{equation}
here, we used $2\le p$ and $2<\frac{2(p+1)}{p}\le p+1<\frac{2N}{N-2}$.

\textbf{Case 2.} $p<2$ : In this case,
\begin{align*}
  &\left|\pow_p(W_h + \phi)  -  \left(   {{W}_{h}}\right)^p - p \left(   {{W}_{h}}\right)^{p-1}\phi\right|
= \left|p\phi\int_0^1 |W_h + s\phi|^{p-1}-  \left(   {{W}_{h}}\right)^{p-1}\:ds\right| \\
&\le \left\{
\begin{aligned}
&p|\phi|^p\left| \int_{0}^{1}   \left(\frac{|   {{W}_{h}}|}{|\phi|}+s \right)^{p-1}+  \left(\frac{|   {{W}_{h}}|}{|\phi|}\right)^{p-1} ds \right|\le p(2^{p-1}+1)|\phi|^p &&\textrm{if}   \ |   {{W}_{h}}|\le|\phi|, \\
&p|\phi|  \int_{0}^{1}  \left|   {{W}_{h}}+s\phi \right|^{p-1} -  \left(   {{W}_{h}}\right)^{p-1}  ds \le |\phi|^p &&\textrm{if}   \ 0\le \phi\le    {{W}_{h}}, \\
&p|\phi|  \int_{0}^{1}\left|  \left(   {{W}_{h}}+s\phi +s|\phi|\right)^{p-1} -|   {{W}_{h}}+s\phi|^{p-1} \right|ds  \le |\phi|^p &&\textrm{if}  \ -   {{W}_{h}}\le \phi\le 0, \end{aligned}\right.
\end{align*}
%
%
where we used the inequality $(a+b)^{p-1}\le a^{p-1}+ b^{p-1}$ for any $a,b\ge 0$  using the condition $p-1\in(0,1)$.
Then   there is a constant  $c_3>0$, independent of $\ve, h>0$, satisfying  \begin{equation}\label{2pp}
 \begin{aligned}
&{\int_{ \mathbb{R}^N} } |\psi| \left|  \pow_p\left(   {{W}_{h}}+\phi\right)  - \left(   {{W}_{h}}\right)^p - p \left(   {{W}_{h}}\right)^{p-1}\phi \right|dy
 \le c_3\|\psi\|\|\phi\|^p.
\end{aligned}\end{equation}

Finally, Lemma \ref{ksum} implies that    {for  ${\eta}_{1}\in\left(1,\min\{p,2\}\right)$ and ${\eta}_{2}\in\left(\frac{1}{p},1\right)$,} there are constants $c_4, c_5, c_6 >0$, independent of $\ve, h>0$, satisfying

\begin{equation}
 \begin{aligned}\label{atleast2}&  {\int_{ \mathbb{R}^N} } |\psi| \big|  \left(   {{W}_{h}}\right)^p  - \sum_{i=1}^4 \left(U_{h,i}\right)^p \big|dy
  \\&\le c_4 \|\psi\|_{L^2(\mathbb{R}^N)}\Big\|  (U_{h,1}+ \sum_{i=2}^4 U_{h,i} )^p - U_{h,1}^p- \sum_{i=2}^4 \left(U_{h,i}\right)^p  \Big\|_{L^2({\C}_1)}
\\&\le c_4 \|\psi\| \Big\{\Big \|p (U_{h,1}+ \lambda \sum_{i=2}^4 U_{h,i})^{p-1}  \sum_{i=2}^4 U_{h,i} \Big\|_{L^2({\C}_1)}+\Big\|\sum_{i=2}^4 \left(U_{h,i}\right)^p\Big\|_{L^2({\C}_1)}\Big\}
\\&
\le c_5 \|\psi\| \Big(\Big\| e^{-(p-1)|y-ht_1|-(1-{\eta}_{1})|y-ht_1|} e^{-\sqrt{2}{\eta}_{1}h} \Big\|_{L^2({\C}_1)}
\\&\ +\Big\|e^{-(1-{\eta}_{2})p|y-ht_1|}e^{-\sqrt{2}{\eta}_{2}ph}\Big\|_{L^2({\C}_1)}\Big) \\&
\le c_6  \|\psi\| \left( e^{-\sqrt{2}{\eta}_{1}h} +e^{-\sqrt{2}   {{\eta}_{2}}p h}\right), \end{aligned}\end{equation}where $\lambda\in(0,1)$ in the second inequality appears by Taylor's theorem.

In view of \eqref{es}-\eqref{atleast2}, we have that     {if $\phi\in B_{\ve,h}$, for any ${\sigma}\in\left(\sqrt{\frac{2}{3}},1\right)$,  ${\eta}_{1}\in\left(1,\min\{p,2\}\right)$, and $\eta_2\in\left(\frac{1}{p},1\right)$}
\begin{equation} \begin{aligned} \label{p2}
&\|F_{\ve,h}(\phi) \|\\ &\le C_{{\sigma}}  \left(  e^{-{\sigma} {\sqrt{3} h}}+\frac{\ve}{     {{h}^m} }+\|\phi\|^{\min\{2,p\}} +e^{-\sqrt{2} {\eta}_{1}   h}    +e^{-\sqrt{2}   {{\eta}_{2} p} h} \right)
\\&<C_{{\sigma}}  \left(  e^{-{\sigma} {\sqrt{3} h}}  +e^{-\sqrt{2} {\eta}_{1}   h}    +e^{-\sqrt{2}   {{\eta}_{2} p} h} \right)+\frac{\ve^{\gamma}}{2}
\\&<\frac{e^{-\sqrt{2}    h}}{2}+\frac{\ve^{\gamma}}{2}\le\ve^{\gamma}, \end{aligned}
\end{equation}
   {where $C_{{\sigma}}>0$ is a constant, independent of $\ve, h>0$.
Therefore, we complete the proof of the claim \eqref{claimp}.}

\medskip

\textit{Step 2.} We claim that  $F_{\ve,h}$ is a contraction in the ball $B_{\ve,h}$.
By \eqref{defg} and  similar  estimations in \eqref{p222}-\eqref{2pp}, we have
\begin{align*}
& \left|{g}_{\ve,h}(\phi_1 ) -  {g}_{\ve,h}(\phi_2)\right|\\
&= \left| \pow_p\left(   {{W}_{h}}+\phi_2\right)  -\pow_p\left(   {{W}_{h}}+\phi_1\right) + p \left(   {{W}_{h}}\right)^{p-1}\phi_1  - p \left(   {{W}_{h}}\right)^{p-1}\phi_2\right|\\
&= \left|p(\phi_2-\phi_1)\int_{0}^1\left|W_h+s\phi_2+(1-s)\phi_1\right|^{p-1}-(W_h)^{p-1} \:ds \right|\\
&\le c_7\left\{
\begin{aligned}
&\left(\sum_{i=1}^2(|\phi_i|+|\phi_i|^{p-1})\right)|\phi_1-\phi_2|\ &&\textrm{if} \ 2\le p,\\
&\left(\sum_{i=1}^2 |\phi_i|^{p-1} \right)|\phi_1-\phi_2| \ &&\textrm{if} \ p<2,
 \end{aligned}\right.\\
\end{align*}
where $c_7>0$ is a constant, independent of $\ve, h>0$.
   {By  Lemma \ref{inu0} and similar  estimations in \eqref{p222}-\eqref{2pp}, we obtain that  if $\phi_1, \phi_2\in B_{\ve,h}$ and  $\ve>0$ is sufficiently small, then }
\begin{align*}
  \| F_{\ve,h}(\phi_1 ) - F_{\ve,h}(\phi_2)\|
 &\le   {\frac{1}{\rho_0}} \sup_{\|\psi\|=1} \Bigg\{   {\int_{ \mathbb{R}^N}  }|\psi| \left| {g}_{\ve,h}(\phi_1 ) -  {g}_{\ve,h}(\phi_2)\right|dy\Big\}
\\&    {\le C(\sum_{i=1}^2\|\phi_i\|^{\min\{1,p-1\}}) \|\phi_1-\phi_2\|} \le  \frac{1}{2} \| \phi_1 - \phi_2 \|,
\end{align*}where $C>0$ is a constant, independent of $\ve, h>0$.

\medskip

By the above arguments in Steps 1 and 2, we complete the proof of Proposition  \ref{Fixed}.
\end{proof}

We define the energy functional for \eqref{maineq} such that
\begin{equation}
\begin{aligned}\label{energy}
u \mapsto \int_{\mathbb{R}^{N}} \frac{1}{2}\left( |\nabla u|^{2}+{V_{\varepsilon}}(y) u^{2}\right)  -\frac{1}{p+1}  | u|^{p+1}  d y,
\end{aligned}
\end{equation}
and the one restricted in $H_s$ is denoted by $I_\ve(u)$. Notably, $I_{\ve} \in C^{2}\left(H_{s} \right)$. Let $u^*_{\ve}$ be a critical point of $I_{\ve}$ in $H_s$. We first claim that $u^*_\ve$ solves \eqref{maineq}, i.e.,
for any function $\psi\in H^1(\mathbb{R}^{N})$, we have
\begin{equation*}
\begin{aligned}  \int_{ \mathbb{R}^N}\left( \nabla u^*_\ve\cdot\nabla \psi + V_\ve u^*_\ve\psi - |u^*_\ve|^{p-1}u^*_\ve\psi\right)dy=0.
\end{aligned}
\end{equation*}
Moser iteration (see example, \cite{Byeon}) and $W^{2,2}$ estimation (see \cite[Theorem 8.8]{GT}) yield that $u^*_\ve$ is smooth, and thus the critical point $u^*_\ve$ is a smooth solution of \eqref{maineq}. To this ends, we observe that for any function $\psi_o\in C^\infty_0(\mathbb{R}^{N})$, which is odd in $y_n, 3<n\leq N$, we obtain
\begin{equation*}
\begin{aligned}  \int_{ \mathbb{R}^N}\left( \nabla u^*_\ve\cdot\nabla\psi_o + V_\ve u^*_\ve\psi_o - |u^*_\ve|^{p-1}u^*_\ve\psi_o\right)dy=0.
\end{aligned}
\end{equation*}
For any function $\psi_e\in C^\infty_0(\mathbb{R}^{N})$, which is even in $y_n, 3<n\leq N$, consider the symmetrization $\hat{\psi}(y)=\sum^{12}_{i=1}\psi_e(T_i y)$ so that
\begin{equation*}
\begin{aligned}  \int_{ \mathbb{R}^N}\left( \nabla u^*_\ve\cdot\nabla\hat{\psi} + V_\ve u^*_\ve\hat{\psi} - |u^*_\ve|^{p-1}u^*_\ve\hat{\psi}\right)dy=0.
\end{aligned}
\end{equation*}
From symmetry of $u^*_\ve$ and change of variable $T_iy = z$ by each isometry $T_i$
\begin{equation*}
\begin{aligned}  &0
=\sum^{12}_{i=1}\int_{ \mathbb{R}^N} \nabla u^*_\ve(y)\cdot\nabla\psi_e(T_i y) + V_\ve(|y|) u^*_\ve(y)\psi_e(T_i y) - |u^*_\ve(y)|^{p-1}u^*_\ve(y)\psi_e(T_i y)dy\\
&=12\int_{ \mathbb{R}^N}\nabla u^*_\ve(z)\cdot\nabla\psi_e(z) + V_\ve u^*_\ve(z)\psi_e(z) - |u^*_\ve(z)|^{p-1}u^*_\ve(z)\psi_e(z)dy.
\end{aligned}
\end{equation*}
Now the problem to find a solution of \eqref{maineq} reduces to the problem to find a critical point if $I_\ve$ in $H_s$.

From Proposition \ref{Fixed}, for each $h\in S_\ve$ we have  $u_{\ve,h}\in H_s$ satisfying  $\frac{\partial I_{\ve}|_{E_h}(u_{\ve,h})}{\partial u}=0$, where ${u_{\ve,h}}=  {W_{h}}+\phi_{\ve, h}$,  $\phi_{\ve,h} \in E_h$ is the fixed point of $F_{\ve,h}$.
This implies that for each $h\in S_\ve$ there is a Lagrange multiplier $\Lambda_{h}\in\mathbb{R}$ satisfying
\begin{equation}\label{La}
 \int_{\mathbb{R}^N} \nabla {u_{\ve,h}}\cdot \nabla\psi- V_{\ve}{u_{\ve,h}}\psi + |{u_{\ve,h}}|^{p-1}u_h\psi \:dy = \int_{\mathbb{R}^N}\Lambda_{h}\left( \sum_{i=1}^4  U_{h,i}^{p-1} \frac{\partial U_{h,i}}{\partial h}\right) \psi \: dy,
\end{equation}
{for all $\psi \in H_s$}.
We define the function
\begin{equation*}\label{def_F}
\mathfrak{F}_{\ve}:~~ S_\ve \ni h ~~ \mapsto ~~  I_{\ve}(u_{\ve,h}).
\end{equation*}
We recall that $\|\phi_h\|\le \ve^\gamma$, $\phi_h\in E_h$, and the equation for $\frac{\partial W_h}{\partial h}$ in \eqref{peq_wh}.
Then $\mathfrak{F}_{\ve}^{\prime}(h^*)=0$ implies that $u_{\ve,h^*}$ is a critical point of $I_\ve$ in $H_s$. 
To complete the proof of Theorem \ref{mainthm}, it suffices to show that the maximization problem $\displaystyle\max _{h \in {S}_{\ve}} \mathfrak{F}_{\ve}(h)$ is achieved by an interior point of ${S}_{\ve}$.

To consider the maximization problem, we first recall from \eqref{l0} that
\begin{equation*}
\begin{aligned} \left<L_{\varepsilon, h}(\phi),\psi\right> = \int_{ \mathbb{R}^N}\left( \nabla \phi\cdot\nabla\psi + V_\ve\phi\psi - p  \left(   {{W}_{h}}\right)^{p-1} \phi\psi\right)dy,
\end{aligned}
\end{equation*}
and define \begin{equation*}l_{\varepsilon,h}(\phi):=\int_{\mathbb{R}^N} (V_{\ve}(y)-1){W}_{h} \phi+\left(\sum_{i=1}^4 (U_{h,i})^p- ({W}_{h})^p\right)  \phi dy,\end{equation*} and
\begin{equation*}
\begin{aligned}
R(\phi)&:=\frac{1}{p+1}\int_{\mathbb{R}^{N}}  ({W}_{h})^{p+1}+  (p+1)({W}_{h})^p\phi+\frac{p(p+1)}{2}({W}_{h})^{p-1} (\phi)^{2}  dy\\&\   -\frac{1}{p+1}\int_{\mathbb{R}^{N}}  ({W}_{h}+\phi)^{p+1}   d y.
\end{aligned}
\end{equation*}
 Then we note that
\begin{equation}\begin{aligned}
\mathfrak{F}_{\ve}(h)&=I_{\ve}\left( {W_{h}}\right)+l_{\varepsilon,h}(\phi_h)+\frac{1}{2}<L_{\varepsilon,h}(\phi_h),\phi_{h}>+R\left(\phi_{h}\right).
\end{aligned}\end{equation}

We show that $I_{\ve}\left( {W_{h}}\right)$ is the leading order contribution as $\ve \to 0$ and is the only relevant term for the maximization problem.
\begin{lemma}\label{error}There is a constant  $
c_{{\sigma}}>0$, independent of $\ve, h>0$, satisfying
\begin{equation*}\begin{aligned}&\left|l_{\varepsilon,h}(\phi_h)\right|+ \frac{1}{2} \left|<L_{\varepsilon,h}(\phi_h),\phi_{h}>\right|+\left|R\left(\phi_{h}\right)\right|\\&\le c_{{\sigma}}  \Big( \frac{\ve^{2}}{ {h}^{2m } }+e^{-2{\sigma} {\sqrt{3} h}}
+e^{-2\sqrt{2}   {{\eta}_{1}} h}
+e^{-2\sqrt{2}   {{\eta}_{2}p} h}\Big),\end{aligned}\end{equation*}
where  ${\sigma} \in\left(\sqrt{\frac{2}{3}},1\right)$,    ${\eta}_{1}\in\left(1,\min\{p,2\}\right)$, and  ${\eta}_{2}\in\left(\frac{1}{p},1\right)$.
\end{lemma}
\begin{proof}
From the estimations \eqref{1e}, \eqref{2e}, and \eqref{atleast2}, we obtain that for ${\sigma} \in\left(\sqrt{\frac{2}{3}},1\right)$,    ${\eta}_{1}\in\left(1,\min\{p,2\}\right)$, and  ${\eta}_{2}\in\left(\frac{1}{p},1\right)$
\begin{equation}\begin{aligned} \label{1ee}
 &\left|l_{\varepsilon,h}(\phi_h)\right|   \le C_{{\sigma},\eta_1,{\eta}_{2}} \|\phi_h\| \Big\{ \frac{\ve}{ {h}^{m } }+e^{-{\sigma} h\sqrt{3}}
+e^{-\sqrt{2}   {{\eta}_{1} } h}
+e^{-\sqrt{2}   {{\eta}_{2} p} h}\Big\},
\end{aligned}\end{equation}where $ C_{{\sigma},\eta_1,{\eta}_{2}} >0$ is a constant, independent of $\ve, h>0$.

By Lemma \ref{inu00},  there is a constant $C>0$, independent of $\ve, h>0$, satisfying
\begin{equation*}\begin{aligned} \label{2ee}
 &\left|<L_{\varepsilon,h}(\phi_h),\phi_{h}>\right|   \le C\|\phi_h\|^2.
\end{aligned}\end{equation*}

Moreover,  Using similar arguments in \eqref{p222} and \eqref{p22}, we obtain that if $\ve>0$ is sufficiently small, there are constants $C_1, C_2>0$, independent of $\ve, h>0$, satisfying
\begin{equation*}\begin{aligned} \label{3ee}
 & \left|R\left(\phi_{h}\right)\right|  \le C_1\left(\|\phi_h\|^2+\|\phi_h\|^{p+1}\right)\le C_2 \|\phi_h\|^2.
\end{aligned}\end{equation*}

Notably, if $\phi_h$ is a fixed point of $F_{\ve,h}$, \eqref{p2} implies that  there is a  constant $c_{{\sigma}}'>0$, independent of $\ve, h>0$, satisfying

\begin{equation}\begin{aligned} \label{3eee}
 &  \|\phi_h\|=\|F_{\ve,h}(\phi_h) \| \le c_{{\sigma}}' \Big\{ \frac{\ve}{ {h}^{m } }+e^{-{\sigma} {\sqrt{3} h}}
+e^{-\sqrt{2}   {{\eta}_{1}} h}
+e^{-\sqrt{2}   {{\eta}_{2}p} h}\Big\},
\end{aligned}\end{equation}

From \eqref{1ee}-\eqref{3eee}, we complete the proof of Lemma \ref{error}.
\end{proof}

Now we expand the main term of the energy functional.
\begin{proposition}\label{A.3} We have as $\ve\to0$,
\begin{equation*}\begin{aligned}
I_{\ve}\left( {W_{h}}\right)
&= \frac{2(p-1)}{(p+1)} \int_{\mathbb{R}^{N}} (U_0)^{p+1}dy+ \frac{2a\ve}{|ht_1|^m}\int_{\mathbb{R}^N}U_0^2dy
 - \mathbf{J_{*}}(h)h^{-\frac{N-1}{2}}e^{-2\sqrt{2}h}\\&+o\left(\frac{\ve}{h^m}\right)+o\left(h^{-\frac{N-1}{2}}e^{-2\sqrt{2}h}\right),\end{aligned}
\end{equation*}
where     $\mathbf{J_{*}}(h)$ satisfies
$0< B_0\le \mathbf{J_{*}}(h)\le B_1$ {for some constants $B_0$ and $B_1$, which  are independent of $\ve, h>0$.}
\end{proposition}

Before we prove the Proposition \ref{A.3}, we show the following Lemma.
\begin{lemma}\label{lemma_mu} If $h \in S_{\ve}$, then
\[\int_{\mathbb{R}^N}({V}_{\ve}(y)-1) U_{h,1}^{2}dy=\frac{a\ve}{|ht_1|^m}\int_{\mathbb{R}^N}U_0^2dy+O\left(\frac{\ve}{{h}^{m+1}}+\frac{\ve}{{h}^{m+\theta}}+e^{-2{\sigma} {\sqrt{3} h}}\right),\] where ${\sigma}\in\left(\sqrt{\frac{2}{3}},1\right)$. \end{lemma}
\begin{proof}Recall that $V_\ve(y)-1=V_1(y)$. Fix a constant ${\sigma}\in\left(\sqrt{\frac{2}{3}},1\right)$.
We see that for sufficiently small $\ve>0$, there are constants $c_1,c_2>0$, independent of $\ve>0$, satisfying
\begin{align*}
 \left|{V}_1(y+{ht_1})-\frac{a\ve}{|ht_1|^m}\right| &\le a\ve \left| \frac{1}{|y+{ht_1}|^m} - \frac{1}{|ht_1|^m}\right| + c_1 \frac{ \ve}{|y+{ht_1}|^{   {m+\theta}}} \\
 &\le c_2\ve\left(\frac{|y|}{|ht_1|^{m+1}} + \frac{1}{|y+{ht_1}|^{   {m+\theta}}}\right)\quad\textrm{for}\ \ y\in  B_{{\sigma} h |t_1|}(0).
\end{align*}
Therefore, we have
\begin{equation*}\label{mu1}
\begin{aligned}
&\int_{\mathbb{R}^N}\left({V}_{\ve}-1\right) U_{h,1}^{2}dy -\frac{a\ve}{|ht_1|^m}\int_{\mathbb{R}^N}U_0^2dy\\& =  \ve\int_{\mathbb{R}^N}\left({V}_{1}(y+{ht_1})- \frac{a}{|ht_1|^m}\right)U_0^2 \:dy\\
&= \ve\int_{B_{{\sigma} h |t_1|}(0)}\left({V}_{1}(y+{ht_1})- \frac{a}{|ht_1|^m}\right)U_0^2 \:dy + \ve O\left(e^{-2{\sigma} h |t_1|}\right)\\& =  O\left(\frac{\ve}{{h}^{m+1}}+\frac{\ve}{{h}^{m+\theta}}+\ve e^{-2{\sigma} {\sqrt{3} h}}\right).
\end{aligned}
\end{equation*}

\end{proof}

Now, we prove  Proposition \ref{A.3}.
\begin{proof}[Proof of Proposition \ref{A.3}] By the definition of $I_\ve$ in \eqref{energy} and the equation \eqref{eq_wh}, we see that 
\begin{align*}
 I_{\ve}\left( {W_{h}}\right) &=
  \int_{\mathbb{R}^{N}} \frac{\left|\nabla  W_h  \right|^{2}+{V}_{\ve}(y) \left(W_h\right)^{2}}{2}
  -\frac{ \left( W_h \right)^{p+1}}{p+1} d y\\
 &=
  \frac{1}{2} \int_{\mathbb{R}^{N}}\left({V}_{\ve}-1\right) \left(W_h\right)^{2}  d y\\& +\frac{1}{2}\int_{\mathbb{R}^N} \Big( \sum_{i=1}^4 (U_{h,i})^p \Big)W_h d y-\frac{1}{p+1}\int_{\mathbb{R}^N} \left(W_h\right)^{p+1} d y.
 \end{align*}By using $W_h=\sum_{i=1}^4 (U_{h,i})$, we note that
\begin{align*}
&\int_{\mathbb{R}^{N}}({V}_{\ve}-1) \left( W_h\right)^{2}dy\\&=\int_{\mathbb{R}^{N}}({V}_{\ve}-1) \sum_{i=1}^4 (U_{h,i})^2 \:dy +  \int_{\mathbb{R}^N}({V}_{\ve}-1)  \sum_{i\ne j} U_{h,i}U_{h,j}\:dy.
\end{align*}
Because $V$ and $U_0$ are radial symmetric functions, we obtain from Lemma \ref{lemma_mu} that  for $1\le i\le 4$,
\begin{align*}
& \int_{\mathbb{R}^{N}}({V}_{\ve}-1)  \sum_{i=1}^4 (U_{h,i})^2 dy \\&=4\int_{\mathbb{R}^{N}}({V}_{\ve}-1) \left( U_{h,1}\right)^{2} \:dy \\&=\frac{4a\ve}{|ht_1|^m}\int_{\mathbb{R}^N}U_0^2dy+O\left(\frac{\ve}{{h}^{m+1}}+\frac{\ve}{{h}^{m+\theta}}+e^{-2{\sigma} {\sqrt{3} h}}\right),
\end{align*}where ${\sigma}\in\left(\sqrt{\frac{2}{3}},1\right)$.
Using the radial symmetric property of $V$ and $U_0$, and Lemma \ref{ksum}, there are constants $c_1,c_2,c_3>0$, independent of $\ve, h>0$, satisfying
\begin{align*}
& \left|\int_{\mathbb{R}^N}({V}_{\ve}-1)  \sum_{i\ne j} U_{h,i}U_{h,j}\:dy \right| =  4\left|\int_{{\C}_1}({V}_{\ve}-1)  \sum_{i\ne j} U_{h,i}U_{h,j}\:dy \right|\\&\le c_1 \int_{{\C}_{1}}|{V}_{\ve}(y)-1|  \Big(\sum_{j= 2}^4U_{h,j}\Big)\Big(\sum_{i=1}^4 U_{h,i}\Big)\:dy\\
&\le c_2 e^{-\sqrt{2}{\eta}_{3} h}  \int_{{\C}_{1}}|{V}_{\ve}(y)-1|  e^{-(2-{\eta}_{3})|y-{ht_1}|} \:dy\\
&\le c_3 e^{-\sqrt{2}{\eta}_{3} h} \left(\frac{\ve}{{h}^m}+ e^{-(2-{\eta}_{3}){\sigma} {\sqrt{3} h}} \right) \quad \text{for some ${\eta}_{3} \in (0,2)$ and ${\sigma}\in\left(\sqrt{\frac{2}{3}},1\right)$},
\end{align*}
whhere, in the last inequality, we used the arguments in  the proof of Lemma \ref{lemma_mu} for $e^{-(2-{\eta}_{3})|y-ht_1|}$ in place of $(U_{h,1})^2$.

Moreover, using  the radial symmetry of $U_0$, we   obtain that
\begin{align*}
&\frac{1}{2}\int_{\mathbb{R}^{N}} \Big( \sum_{i=1}^{4} ( U_{h,i})^{p}\Big) {W}_{h}dy -\frac{1}{p+1}\int_{\mathbb{R}^{N}} \left(W_h\right)^{p+1}\:dy\\
&=\left(\frac{1}{2}-\frac{1}{p+1}\right)\int_{\mathbb{R}^{N}}  \sum_{i=1}^{4} ( U_{h,i})^{p+1}  dy
\\& +\frac{1}{2}\int_{\mathbb{R}^{N}} \Big( \sum_{i=1}^{4} ( U_{h,i})^{p}\Big) {W}_{h}-  \sum_{i=1}^{4} ( U_{h,i})^{p+1} dy
\\&  -\frac{1}{p+1}\int_{\mathbb{R}^{N}} \left(W_h\right)^{p+1}-   \sum_{i=1}^{4} ( U_{h,i})^{p+1} \:dy
\\&=  \frac{2(p-1)}{p+1}\int_{\mathbb{R}^{N}}   (U_{0})^{p+1}  dy  -2\int_{{\C}_1}   ( U_{h,1})^{p}\Big( \sum_{i=2}^{4}  U_{h,i} \Big)dy +K_{h}, \end{align*}
where
\begin{align*} K_h &:= 2\int_{{\C}_1}    \big( \sum_{i=2}^{4} ( U_{h,i})^{p}\big) {W}_{h}-  \sum_{i=2}^{4} ( U_{h,i})^{p+1} dy
 \\&-4\int_{{\C}_1}   \frac{ p}{2}\big(U_{h,1}+\lambda\sum_{i=2}^4 U_{h,i}\big)^{p-1}\big( \sum_{i=2}^{4}U_{h,i} \big)^2-  \frac{1}{p+1} \sum_{i=2}^{4} ( U_{h,i})^{p+1} \:dy,\end{align*} here, $\lambda\in(0,1)$.
 By  Lemma \ref{ksum}, we see that for ${\eta}_{4}\in\left(\frac{2}{p},1+\frac{1}{p}\right)$,  ${\eta}_{5}\in\left(\frac{2}{p+1},1\right)$, and ${\eta}_{6}\in\left(1, \min \{\frac{p+1}{2},2\}\right)$,
\begin{align*}   &|K_h|\\&\le c_1 \int_{{\C}_1} \big( \sum_{i=2}^{4}   U_{h,i} \big)^{p }  \big(\sum_{i=1}^4 U_{h,i}\big)+   \big( \sum_{i=2}^{4}  U_{h,i} \big)^{p+1}  +\big(\sum_{i=1}^4 U_i\big)^{p-1}\big( \sum_{i=2}^{4}U_{h,i} \big)^2   dy
\\&\le c_2\Big( \int_{{\C}_1} e^{-\sqrt{2}{\eta}_{4}ph}e^{-\{(1-{\eta}_{4})p+1\}|y-ht_1|} +e^{-\sqrt{2}{\eta}_{5}(p+1)h}e^{-(1-{\eta}_{5})(p+1)|y-ht_1|}\\&+e^{-2\sqrt{2}{\eta}_{6}h} e^{-\{(p-1)+2(1-{\eta}_{6})\}|y-ht_1|}dy\Big)
\\&\le c_3\left( e^{-\sqrt{2}{\eta}_{4}ph} +e^{-\sqrt{2}{\eta}_{5}(p+1)h}+e^{-2\sqrt{2}{\eta}_{6}h}\right).
\end{align*}
 In summary,  we have
\begin{align*}
I_{\ve}\left( {W_{h}}\right)
 &= \frac{2(p-1)}{p+1}\int_{\mathbb{R}^{N}}   (U_{0})^{p+1}  dy  -2\int_{{\C}_1}   ( U_{h,1})^{p}\Big( \sum_{i=2}^{4}  U_{h,i} \Big)dy\\&+\frac{2a\ve}{|ht_1|^m}\int_{\mathbb{R}^N}U_0^2dy+O\left(\frac{\ve}{{h}^{m+1}}+\frac{\ve}{{h}^{m+\theta}}+e^{-2{\sigma} {\sqrt{3} h}}\right)
 \\&+O\left(e^{-\sqrt{2}{\eta}_{3} h}\frac{\ve}{{h}^m}+ e^{-\sqrt{2}{\eta}_{3} h -(2-{\eta}_{3}){\sigma} {\sqrt{3} h}} \right)
 \\&+O\left( e^{-\sqrt{2}{\eta}_{4}ph} +e^{-\sqrt{2}{\eta}_{5}(p+1)h}+e^{-2\sqrt{2}{\eta}_{6}h}\right),  \end{align*}where ${\sigma}\in\left(\sqrt{\frac{2}{3}},1\right)$,   ${\eta}_{3} \in (0,2)$, ${\eta}_{4}\in\left(\frac{2}{p},1+\frac{1}{p}\right)$,  ${\eta}_{5}\in\left(\frac{2}{p+1},1\right)$, and ${\eta}_{6}\in\left(1, \min \{\frac{p+1}{2},2\}\right)$.
Finally, we observe that $|ht_{1}-ht_{i}|=2\sqrt{2}h$, $i=2,3,4$, and
\begin{align*}
&2\int_{{\C}_1}   ( U_{h,1})^{p}\Big( \sum_{i=2}^{4}  U_{h,i} \Big)dy\\
& = 2\sum_{i=2}^{4} \left(e^{-\left|ht_{1}-ht_{i}\right|}|{ht_1}-{ht_i}|^{-\frac{N-1}{2}}\right) \int_{{\C}_1} (U_{h,1})^{p} U_{h,i}  \left(e^{\left|ht_{1}-ht_{i}\right|}|{ht_1}-{ht_i}|^{\frac{N-1}{2}}\right) \:dy
\\&=2e^{-2\sqrt{2}h}\left(2\sqrt{2}h\right)^{-\frac{N-1}{2}} \sum_{i=2}^{4}\int_{{\C}_1} (U_{h,1})^{p} U_{h,i}  \left(e^{\left|ht_{1}-ht_{i}\right|}|{ht_1}-{ht_i}|^{\frac{N-1}{2}}\right) \:dy\\&= \mathbf{J_{*}}(h)h^{-\frac{N-1}{2}}e^{-2\sqrt{2}h}, \end{align*}
  where    \[\mathbf{J_*}(h)  = 2 \left(2\sqrt{2} \right)^{-\frac{N-1}{2}} \sum_{i=2}^{4}\int_{{\C}_1} (U_{h,1})^{p} U_{h,i}  \left(e^{\left|ht_{1}-ht_{i}\right|}|{ht_1}-{ht_i}|^{\frac{N-1}{2}}\right) \:dy.\]Now we claim that there are constants  $B_0,B_1>0$,  independent of $\ve, h>0$, satisfying   \begin{equation*}\label{J}0< B_0\le \mathbf{J_{*}}(h)\le B_1.\end{equation*} Because $B_{\sqrt{h}}(ht_1)\subseteq {\C}_1$ by the definition of ${\C}_1$ in \eqref{defo} if $\ve>0$ is sufficiently small, from \eqref{exp0}, there are constants $c_1, c_2, c_3>0$, independent of $\ve,h>0$, satisfying
\begin{align*}&\mathbf{J_*}(h) \\&\ge  c_1 \sum_{i=2}^4\int_{B_{\sqrt{h}}(ht_1)} (U_{h,1})^{p} U_{h,i}  \left(e^{\left|ht_{1}-ht_{i}\right|}|{ht_1}-{ht_i}|^{\frac{N-1}{2}}\right) \:dy\\&= c_1 \sum_{i=2}^4 \int_{B_{\sqrt{h}}(0)} (U_{0}(y))^{p} U_{0}(y+ht_1-ht_i)  \left(e^{\left|ht_{1}-ht_{i}\right|}|{ht_1}-{ht_i}|^{\frac{N-1}{2}}\right) \:dy
\\&\ge  c_2 \sum_{i=2}^4 \int_{B_{\sqrt{h}}(0)} (U_{0}(y))^{p}e^{-|y+ht_1-ht_i|+|ht_{1}-ht_{i}|}\left(\frac{|{ht_1}-{ht_i}|^{\frac{N-1}{2}}}{|y+ht_1-ht_i|^{\frac{N-1}{2}}}\right) \:dy
\\&\ge  c_3  \int_{B_{\sqrt{h}}(0)} (U_{0}(y))^{p}e^{-|y|}  \:dy \ge   c_3 \int_{B_{1}(0)} (U_{0}(y))^{p}e^{-|y|}  \:dy>0.    \end{align*} Moreover, we observe that  there are constants $C_1, C_2, C_3>0$, independent of $\ve, h>0$, satisfying
\begin{align*}\mathbf{J_*}(h) &\le C_1 \sum_{i=2}^4\int_{\mathbb{R}^N}  (U_{0}(y))^{p} U_{0}(y+ht_1-ht_i)  \left(e^{\left|ht_{1}-ht_{i}\right|}|{ht_1}-{ht_i}|^{\frac{N-1}{2}}\right) \:dy
\\&\le  C_2 \sum_{i=2}^4  \int_{\mathbb{R}^N} \Big\{ (U_{0}(y))^{p-1}e^{-|y|-|y+ht_1-ht_i|  +|ht_{1}-ht_{i}|}\\&\ \times\left(\frac{|{ht_1}-{ht_i}|}{\left(1+|y|\right) \left(1+|y+ht_1-ht_i|\right) }\right)^{\frac{N-1}{2}} \Big\}\:dy
\\&\le C_3 \int_{\mathbb{R}^N}   (U_{0}(y))^{p-1}\:dy.  \end{align*} From the above arguments, we complete the proof of Proposition \ref{A.3}.
\end{proof}
Now we are ready to prove Theorem \ref{mainthm}.
\begin{proof}[Proof of Theorem \ref{mainthm}]
By Proposition \ref{A.3} and Lemma \ref{error}, we have that as $\ve\to0$,
\begin{equation*}\begin{aligned}
\mathfrak{F}_{\ve}(h)&={A}_{0}+ \frac{2a\ve}{|ht_1|^m}\int_{\mathbb{R}^N}U_0^2dy
 -\mathbf{J_{*}}(h)h^{-\frac{N-1}{2}}e^{-2\sqrt{2}h}\\&+o\left(\frac{\ve}{h^m}\right)+o\left(h^{-\frac{N-1}{2}}e^{-2\sqrt{2}h}\right),\end{aligned}
\end{equation*}where ${A}_{0}=\frac{2(p-1)}{(p+1)} \int_{\mathbb{R}^{N}} (U_0)^{p+1}dy$. We recall $S_\ve= \left[ \left(\frac{1}{2\sqrt{2}}-\beta_0\right)\ln\frac{1}{\varepsilon},  \left(\frac{1}{2\sqrt{2}}+\beta_0\right)\ln\frac{1}{\varepsilon}\right].$   If  $\ve>0$ is  sufficiently small, the value of $\mathfrak{F}_{\ve}$ on $\partial S_\ve$ is less than  the value of $\mathfrak{F}_{\ve}$ at some interior point of $S_{\ve}$. Indeed, we have for  small $\ve>0$,  \[\mathfrak{F}_{\ve}\Big(\Big(\frac{1}{2\sqrt{2}}-\beta_0\Big)\ln\frac{1}{\varepsilon}\Big)< {A}_{0}<  \mathfrak{F}_{\ve}\Big(\Big(\frac{1}{2\sqrt{2}}+\beta_0\Big)\ln\frac{1}{\varepsilon}\Big)<  \mathfrak{F}_{\ve}\Big(\Big(\frac{1}{2\sqrt{2}}+\frac{\beta_0}{2}\Big)\ln\frac{1}{\varepsilon}\Big), \]; thus,      $\max _{h \in S_{\ve}} \mathfrak{F}_{\ve}(h)$ is achieved by an interior point $h_{\ve}$ of $S_{\ve}$. Therefore,  we $u_{h_\ve}$ is a solution to \eqref{maineq},  completing the proof of Theorem \ref{mainthm}.
\end{proof}

\newpage

\appendix
\section{}\label{AA}

\begin {table} [h!]
\begin {center}
\caption {Multiplication table of group $\{T_i \ | \ 1\le i\le 12\}$}
\label {Table 3}
\begin {tabular} {|c|c|c|c|c|c|c|c|c|c|c|c|c|}
\hline
 & $T_1$ & $T_2$ & $T_3$ & $T_4$ &  $T_5$ & $T_6$ & $T_7$ & $T_8$ & $T_9$ & $T_{10}$ &  $T_{11}$ & $T_{12}$ \\ \hline
$T_1$ & $T_1$ & $T_2$ & $T_3$ & $T_4$ &  $T_5$ & $T_6$ & $T_7$ & $T_8$ & $T_9$ & $T_{10}$ &  $T_{11}$ & $T_{12}$ \\ \hline
$T_2$ & $T_2$ & $T_1$ & $T_4$ & $T_3$ & $T_6$ & $T_5$ & $T_8$ & $T_7$ & $T_{10}$ & $T_{9}$ & $T_{12}$ & $T_{11}$ \\ \hline

$T_{3}$ & $T_{3}$ & $T_{4}$ & $T_{1}$ & $T_{2}$ & $T_{7}$ & $T_{8}$ & $T_{5}$ & $T_{6}$ & $T_{11}$ & $T_{12}$ & $T_{9}$ & $T_{10}$
\\ \hline

$T_{4}$ & $T_{4}$ & $T_{3}$ & $T_{2}$ & $T_{1}$ & $T_{8}$ & $T_{7}$ & $T_{6}$ & $T_{5}$ & $T_{12}$ & $T_{11}$ & $T_{10}$ & $T_{9}$
\\ \hline

$T_{5}$ & $T_{5}$ & $T_{8}$ & $T_{6}$ & $T_{7}$ & $T_{9}$ & $T_{12}$ & $T_{10}$ & $T_{11}$ & $T_{1}$ & $T_{4}$ & $T_{2}$ & $T_{3}$ \\ \hline

$T_{6}$ & $T_{6}$ & $T_{7}$ & $T_{5}$ & $T_{8}$ & $T_{10}$ & $T_{11}$ & $T_{9}$ & $T_{12}$ & $T_{2}$ & $T_{3}$ & $T_{1}$ & $T_{4}$
\\ \hline

$T_{7}$ & $T_{7}$ & $T_{6}$ & $T_{8}$ & $T_{5}$ & $T_{11}$ & $T_{10}$ & $T_{12}$ & $T_{9}$ & $T_{3}$ & $T_{2}$ & $T_{4}$ & $T_{1}$
\\ \hline

$T_{8}$ & $T_{8}$ & $T_{5}$ & $T_{7}$ & $T_{6}$ & $T_{12}$ & $T_{9}$ & $T_{11}$ & $T_{10}$ & $T_{4}$ & $T_{1}$ & $T_{3}$ & $T_{2}$
\\ \hline

 $T_{9}$ & $T_{9}$ & $T_{11}$ & $T_{12}$ & $T_{10}$ & $T_{1}$ & $T_{3}$ & $T_{4}$ & $T_{2}$ & $T_{5}$ & $T_{7}$ & $T_{8}$ & $T_{6}$
\\ \hline

 $T_{10}$ & $T_{10}$ & $T_{12}$ & $T_{11}$ & $T_{9}$ & $T_{2}$ & $T_{4}$ & $T_{3}$ & $T_{1}$ & $T_{6}$ & $T_{8}$ & $T_{7}$ & $T_{5}$
\\ \hline

 $T_{11}$ & $T_{11}$ & $T_{9}$ & $T_{10}$ & $T_{12}$ & $T_{3}$ & $T_{1}$ & $T_{2}$ & $T_{4}$ & $T_{7}$ & $T_{5}$ & $T_{6}$ & $T_{8}$
\\ \hline

 $T_{12}$ & $T_{12}$ & $T_{10}$ & $T_{9}$ & $T_{11}$ & $T_{4}$ & $T_{2}$ & $T_{1}$ & $T_{3}$ & $T_{8}$ & $T_{6}$ & $T_{5}$ & $T_{7}$
\\ \hline
\end {tabular}
\end {center}
\end {table}

\section{}\label{B}
\begin{proof}[The proof of $\psi_n\in H_s$]
We  claim that  $\psi_n\in H_s$. Because $\psi_n(y_1,\cdots,y_N)$ is even for $y_n$, $n>3$, it suffices to show that
\begin{equation}\label{appendix_claim} \psi_n(y)=\psi_n(T_i y)\quad\textrm{for all}\ \  i\in\{1,\cdots,12\}\  \ \textrm{and}\  y\in\mathbb{R}^N.\end{equation}
By the definition of $\psi_n$ on ${\C}_i$, $i=2,3,4$  in \eqref{extend}, and $\cup_{i=1}^4{\C}_i=\mathbb{R}^N$,  it suffices to show that $\psi_n(y)=\psi_n(T_i y)$\ for all $i\in\{1,\cdots,12\}$ and   $y\in{\C}_1$.
By $T_i^{-1}=T_i$ for $i=1,2,3,4$, we see that $\psi_n(y)=\psi_n(T_i y)$ for $y\in{\C}_1$ and $i=1,2,3,4$.  By \eqref{defpsi}, we also have $\psi_n(y)=\psi_n(T_5 y)=\psi_n(T_9 y)$ for $y\in {\C}_1$. Recall that $T_1$, $T_5$, $T_9$ restricted on $\mathcal{C}_1$ are automorphisms. In summary, for $y\in \mathcal{C}_1$,
\begin{align*}
 \psi_n(y) &= \psi_n(T_5y) = \psi_n(T_2T_5y) = \psi_n(T_3T_5y) = \psi_n(T_4T_5y)\\
 &=\psi_n(T_9y) = \psi_n(T_2T_9y) = \psi_n(T_3T_9y) = \psi_n(T_4T_9y).
 \end{align*}
\end{proof}

\bibliographystyle{amsplain}

\end{document}